\documentclass[11pt,reqno]{amsart}
\usepackage{graphicx}

\usepackage{amscd,amssymb,amsmath,amsthm}
\usepackage{graphicx}
\usepackage{color}
\usepackage{cite}
\usepackage{nicefrac}
\usepackage{marvosym}
\usepackage{nicefrac}
\usepackage{multirow}

\newtheorem{thm}{Theorem}
\newtheorem{defn}{Definition}
\newtheorem{lemma}{Lemma}
\newtheorem{pro}{Proposition}
\newtheorem{rk}{Remark}

\numberwithin{equation}{section} 

\begin{document}
\title{\textsf{Gonosomal
		algebras and associated discrete-time dynamical systems}}
{}

\author{U.A. Rozikov,  S.K. Shoyimardonov, R. Varro}

	\address{ U.A. Rozikov$^{a,b,c}$\begin{itemize}
		\item[$^a$] V.I.Romanovskiy Institute of Mathematics,  9, Universitet str., 100174, Tashkent, Uzbekistan;
		\item[$^b$] New Uzbekistan University,
		54,  Mustaqillik ave.,    100007,  Tashkent, Uzbekistan;
		\item[$^c$] National University of Uzbekistan,  4, Universitet str., 100174, Tashkent, Uzbekistan.
\end{itemize}}
\email{rozikovu@yandex.ru}

\address{S.\ K.\ Shoyimardonov\\ V.I.Romanovskiy Institute of Mathematics,  9, Universitet str., 100174, Tashkent, Uzbekistan.}
\email {shoyimardonov@inbox.ru}

\address{R.\ Varro\\Institut Montpelli\'{e}rain Alexander Grothendieck, Universit\'{e}
	de Montpellier.  Place Eug\`{e}ne Bataillon
	34090 Montpellier, France.}
\email {richard.varro@univ-montp3.fr}

\begin{abstract}
In this paper we study the discrete-time dynamical systems associated
with gonosomal algebras used as algebraic model in the sex-linked
genes inheritance. We show that the class of gonosomal algebras
is disjoint from the other non-associative algebras usually studied
(Lie, alternative, Jordan, associative power). To each gonosomal
algebra, with the mapping $x\mapsto\frac{1}{2}x^{2}$,
an evolution operator $W$ is associated that gives the state of the offspring population at the birth stage, then from $W$ we define the operator $V$ which gives the frequency distribution of genetic types.
We study discrete-time dynamical systems generated by these two operators, in particular
we show that the various stability notions of the equilibrium
points are preserved by passing from $W$ to $V$. Moreover, for the evolution operators associated with genetic
disorders in the case of a diallelic gonosomal lethal gene we give complete analysis of fixed and limit points of the dynamical systems.

 \end{abstract}
\maketitle



\section{Introduction}

In most bisexual species sex determination systems are based
on sex chromosomes also called gonosomes (or heterochromosomes,
idiochromosomes, heterosomes, allosomes). Gonosomes, unlike autosomes
are not homologous, they are often of different sizes and in
all cases they have two distinct regions:

-- the pseudoautosomal region corresponds to homologous regions
on the two gonosome types, it carries genes present on the two
types of sex chromosomes that are transmitted in the same manner
as autosomal genes;

-- the differential region carries genes that are present only
on one type of gonosome and have no counterpart on the other
type, we say that these genes are sex-linked or gonosomal.

The chromosomal dimorphism in gonosomes induces an asymmetry
in the transmission of gonosomal genes: for example, for a diallelic
gene three genotypes are observed in one sex and only two in
the other and when an allele is recessive it is always expressed
in one sex and one third of cases in the other. Therefore inheritance
of gonosomal genes is very different from that of autosomal genes.

\medskip{}

Population genetics studies the evolution (dynamics) of frequency distributions
of genetic types (alleles, genotypes, gene collections etc.)
in successive generations under the action of evolutionary forces.
This study is based on the definition and application of an evolution
operator to describe the next generation state knowing that
of the previous generation, i.e., the discrete-time dynamical systems generated by the evolution operator (cf. \cite{Ba}, \cite{Lyub-92}, \cite{RShV}, \cite{WB-80}).

The book \cite{Ba} contains a short history of applications of mathematics to
solving various problems in population dynamics. Moreover, in \cite{Lyub-92} for a class of populations a very effective  algebraic-dynamical theory is developed.

In recent book \cite{Rpd} the theory of discrete-time dynamical
systems and evolution algebras of free and sex linked populations are systematically presented.

In this paper we continue the study initiated in \cite{Varro}, \cite{Roz-Va} on gonosomal algebras and discrete-time dynamical systems modeling sex-linked genes
inheritance. Knowing the inheritance coefficients of a bisexual
panmictic population, we define from these coefficients a gonosomal algebra. Next from a gonosomal algebra we define an evolution
operator $W$ called gonosomal operator. The multivariate quadratic
operator $W$ connects the genetic states of two successive generations.
From the operator $W$ we construct an operator $V$ called the
normalized gonosomal operator of $W$, operator $V$ is composed
of multivariate quadratic rational functions, it connects the
frequency distributions of two successive generations. We study
these two operators and we show that the different stability
notions of equilibrium points for $W$ are retained for $V$. In the last section we study the inheritance dynamics of a diallelic
lethal gonosomal gene.

\section{Evolution operators of a bisexual panmictic population}

In a bisexual panmictic population with discrete nonoverlapping
generations, we consider a gonosomal gene whose genetic types
in females (resp. in males) are $\left(e_{i}\right)_{1\text{\ensuremath{\le}}i\text{\ensuremath{\le}}n}$
(resp. $\left(\widetilde{e}_{j}\right)_{1\text{\ensuremath{\le}}j\text{\ensuremath{\le}}\nu}$).\smallskip{}

We note:\smallskip{}

\begin{itemize}
\item $x_{i}^{\:\left(t\right)}$ (resp. $y_{j}^{\:\left(t\right)}$)
the frequency of type $e_{i}$ (resp. $\widetilde{e}_{j}$) in
females (resp. males) born in generation $t\in\mathbb{N}$, so
$x_{i}^{\:\left(t\right)},y_{j}^{\:\left(t\right)}\geq0$ and
$\sum_{i=1}^{n}x_{i}^{\:\left(t\right)}+\sum_{j=1}^{\nu}y_{j}^{\:\left(t\right)}=1$.
\item $\gamma_{ijk}$ (resp. $\widetilde{\gamma}_{ijr}$) the probability
that a female (resp. a male) offspring is of type $e_{k}$ (resp.
$\widetilde{e}_{r}$) when the parental pair is a female of type
$e_{i}$ and a male of type $\widetilde{e}_{p}$, so $\gamma_{ijk},\widetilde{\gamma}_{ijr}\geq0$
and $\sum_{k=1}^{n}\gamma_{ijk}+\sum_{r=1}^{\nu}\widetilde{\gamma}_{ijr}=1.$\smallskip{}

\end{itemize}
After random mating, the proportion in the generation $t+1$
of female (resp. male) type $e_{k}$ (resp. $\widetilde{e}_{r}$)
offsprings born from the crossing between all possible parents
is
\begin{equation}
\sum_{i,j=1}^{n,\nu}\gamma_{ijk}x_{i}^{\:\left(t\right)}y_{j}^{\:\left(t\right)}\quad\Bigl(\mbox{resp. }\sum_{i,j=1}^{n,\nu}\widetilde{\gamma}_{ijr}x_{i}^{\:\left(t\right)}y_{j}^{\:\left(t\right)}\Bigr).\label{eq:Proportions}
\end{equation}
We deduce that the total number $N\left(t+1\right)$ of the population
at generation $t+1$ is
\begin{eqnarray}
N\left(t+1\right) & = & \sum_{k=1}^{n}\sum_{i,j=1}^{n,\nu}\gamma_{ijk}x_{i}^{\:\left(t\right)}y_{j}^{\:\left(t\right)}+\sum_{r=1}^{\nu}\sum_{i,j=1}^{n,\nu}\widetilde{\gamma}_{ijr}x_{i}^{\:\left(t\right)}y_{j}^{\:\left(t\right)}\label{eq:Eff-Tot}\\
 & = & \Bigl(\sum_{i=1}^{n}x_{i}^{\:\left(t\right)}\Bigr)\Bigl(\sum_{j=1}^{\nu}y_{j}^{\:\left(t\right)}\Bigr)\nonumber
\end{eqnarray}
 therefore if $N\left(t+1\right)\neq0$, the frequency of type
$e_{k}$ (resp. $\widetilde{e}_{r}$) in the generation $t+1$
is given by:
\begin{eqnarray}\label{eq:Frequencies}
x_{k}^{\:\left(t+1\right)} & = & \frac{\sum_{i,j=1}^{n,\nu}\gamma_{ijk}x_{i}^{\:\left(t\right)}y_{j}^{\:\left(t\right)}}{\left(\sum_{i=1}^{n}x_{i}^{\:\left(t\right)}\right)\bigl(\sum_{j=1}^{\nu}y_{j}^{\:\left(t\right)}\bigr)}\\
\Bigl(\mbox{resp. }y_{k}^{\:\left(t+1\right)} & = & \frac{\sum_{i,j=1}^{n,\nu}\widetilde{\gamma}_{ijr}x_{i}^{\:\left(t\right)}y_{j}^{\:\left(t\right)}}{\left(\sum_{i=1}^{n}x_{i}^{\:\left(t\right)}\right)\bigl(\sum_{j=1}^{\nu}y_{j}^{\:\left(t\right)}\bigr)}\Bigr).
\end{eqnarray}
Consider $(n+\nu-1)-$dimensional simplex
$$S^{n+\nu-1}=\left\{(x_1, \dots, x_n; y_1, \dots, y_\nu)\in \mathbb R^{n+\nu}: x_i\geq 0, \, y_j\geq 0, \, \sum_{i=1}^{n}x_{i}+\sum_{j=1}^{\nu}y_{j}=1\right\}.$$
Then equations \eqref{eq:Frequencies} is a discrete-time dynamical system generated by the evolution operator $W: S^{n+\nu-1} \to S^{n+\nu-1}$ defined as (see \cite{Roz-Va})
\begin{equation}\label{oW}
W:	\begin{array}{ll}
	x_{k}' =  \frac{\sum_{i,j=1}^{n,\nu}\gamma_{ijk}x_{i}y_{j}}{\left(\sum_{i=1}^{n}x_{i}\right)\bigl(\sum_{j=1}^{\nu}y_{j}\bigr)}\\[4mm]
	y_{k}'= \frac{\sum_{i,j=1}^{n,\nu}\widetilde{\gamma}_{ijr}x_{i}
		y_{j}}{\left(\sum_{i=1}^{n}x_{i}\right)\bigl(\sum_{j=1}^{\nu}y_{j}\bigr)}.
	\end{array}
\end{equation}

\medskip{}

\section{Definition and basic properties of gonosomal algebras}

There are several algebraic models to study the inheritance of
gonosomal genes. The first was proposed by Etherington \cite{Ether-41}
for a gonosomal diallelic gene in the $XY$-system, it was extended
to diallelic case with mutation in \cite{Gonsh-60}, to multiallelic
case in \cite{Gonsh-65,WB-74,WB-75}. The second model is due
to Gonshor \cite{Gonsh-73} by introducing the concept of sex-linked
duplication. In \cite{LR-10} the authors introduced a more general
definition: the evolution algebras of a bisexual population (\emph{EABP}).
In \cite{Varro} we show that several genetic situations are
not representable by \emph{EABP} what leads to put the following
definition.
\begin{defn}
\label{def:Gonosomal-Alg} Given a commutative field $K$ with
characteristic $\neq2$, a $K$-algebra $A$ is gonosomal of
type $\left(n,\nu\right)$ if it admits a basis $\left(e_{i}\right)_{1\text{\ensuremath{\le}}i\text{\ensuremath{\le}}n}\cup\left(\widetilde{e}_{j}\right)_{1\text{\ensuremath{\le}}j\text{\ensuremath{\le}}\nu}$
such that for all $1\leq i,j\leq n$ and $1\leq p,q\leq\nu$
we have:
\begin{eqnarray*}
e_{i}e_{j} & = & 0,\\
\widetilde{e}_{p}\widetilde{e}_{q} & = & 0,\\
e_{i}\widetilde{e}_{p}\;=\;\widetilde{e}_{p}e_{i} & = & \sum_{k=1}^{n}\gamma_{ipk}e_{k}+\sum_{r=1}^{\nu}\widetilde{\gamma}_{ipr}\widetilde{e}_{r},
\end{eqnarray*}
where $\sum_{k=1}^{n}\gamma_{ipk}+\sum_{r=1}^{\nu}\widetilde{\gamma}_{ipr}=1$.
The basis $\left(e_{i}\right)_{1\text{\ensuremath{\le}}i\text{\ensuremath{\le}}n}\cup\left(\widetilde{e}_{j}\right)_{1\text{\ensuremath{\le}}j\text{\ensuremath{\le}}\nu}$
is called a gonosomal basis of $A$.\end{defn}
\begin{rk}
For now, we do not need to assume that the structure constants
$\gamma_{ipk}$, $\widetilde{\gamma}_{ipr}$ are non-negative.
\end{rk}
It was shown in \cite{Varro} that gonosomal algebras can represent
algebraically all sex determination systems ($XY$, $WZ$, $X0$,
$Z0$ and $WXY$) and a wide variety of genetic phenomena related
to sex as{\footnotesize{}:}\emph{\footnotesize{} }temperature-dependent
sex determination, sequential hermaphrodism, androgenesis, parthenogenesis,
gynogenesis, bacterial conjugation, cytoplasmic inheritance,
sex-linked lethal genes, multiple sex chromosome systems, heredity
in the $WXY$-system, heredity in the $WZ$-system with male
feminization, $XY$-system with fertile $XY$-females, $X$-linked
sex-ratio distorter, kleptogenesis,\emph{ }genetic processes
(mutation, recombination, transposition) influenced by sex, heredity
in ciliates, genomic imprinting, $X$-inactivation, sex determination
by gonosome elimination, sexual reproduction in triploid, polygenic
sex determination, cytoplasmic heredity.\medskip{}

The gonosomal basis on a gonosomal algebra may be not unique
as as shown by the following proposition.
\begin{pro}
Let $A$ be a gonosomal algebra with gonosomal basis $\left(e_{i}\right)_{1\text{\ensuremath{\le}}i\text{\ensuremath{\le}}n}\cup\left(\widetilde{e}_{p}\right)_{1\text{\ensuremath{\le}}p\le\nu}$.
Then any basis $\left(a_{i}\right)_{1\text{\ensuremath{\le}}i\text{\ensuremath{\le}}n}\cup\left(\widetilde{a}_{p}\right)_{1\text{\ensuremath{\le}}p\le\nu}$
with
\[
a_{i}=\sum_{j=1}^{n}\alpha_{ji}e_{j}\mbox{ and }\widetilde{a}_{p}=\sum_{q=1}^{\nu}\widetilde{\alpha}_{qp}\widetilde{e}_{p}
\]
where $\sum_{j=1}^{n}\alpha_{ji}=\sum_{q=1}^{\nu}\widetilde{\alpha}_{qp}=1$
for all $1\leq i\leq n,1\leq p\leq\nu$, is a gonosomal basis
of $A$.\end{pro}
\begin{proof}
Let $\left(a_{i}\right)_{1\text{\ensuremath{\le}}i\text{\ensuremath{\le}}n}\cup\left(\widetilde{a}_{p}\right)_{1\text{\ensuremath{\le}}p\le\nu}$
be a basis of the assumed form. It is immediate that $a_{i}a_{j}=\widetilde{a}_{p}\widetilde{a}_{q}=0$.
Next by an easy calculation we get
\begin{eqnarray*}
a_{i}\widetilde{a}_{p} & = & \sum_{k=1}^{n}\bigl(\sum_{j,q=1}^{n,\nu}\alpha_{ji}\widetilde{\alpha}_{qp}\gamma_{jqk}\bigr)e_{k}+\sum_{r=1}^{\nu}\bigl(\sum_{j,q=1}^{n,\nu}\alpha_{ji}\widetilde{\alpha}_{qp}\widetilde{\gamma}_{jqr}\bigr)\widetilde{e}_{r}
\end{eqnarray*}
where
\begin{eqnarray*}
\sum_{k=1}^{n}\bigl(\sum_{j,q=1}^{n,\nu}\alpha_{ji}\widetilde{\alpha}_{qp}\gamma_{jqk}\bigr)+\sum_{r=1}^{\nu}\bigl(\sum_{j,q=1}^{n,\nu}\alpha_{ji}\widetilde{\alpha}_{qp}\widetilde{\gamma}_{jqr}\bigr) & = & \sum_{j,q=1}^{n,\nu}\alpha_{ji}\widetilde{\alpha}_{qp}\bigl(\sum_{k=1}^{n}\gamma_{jqk}+\sum_{r=1}^{\nu}\widetilde{\gamma}_{jqr}\bigr)\\
 & = & \bigl(\sum_{j=1}^{n}\alpha_{ji}\bigr)\bigl(\sum_{q=1}^{\nu}\widetilde{\alpha}_{qp}\bigr)=1,
\end{eqnarray*}
which establishes that the basis $\left(a_{i}\right)_{1\text{\ensuremath{\le}}i\text{\ensuremath{\le}}n}\cup\left(\widetilde{a}_{p}\right)_{1\text{\ensuremath{\le}}p\le\nu}$
is gonosomal.\end{proof}
\begin{pro}\label{prop:Iso-gonosomal-algebra}Any gonosomal algebra of type
$\left(n,\nu\right)$ is isomorphic to a gonosomal algebra of
type $\left(\nu,n\right)$.\end{pro}
\begin{proof}
Let $A$ be a gonosomal algebra with basis $\left(e_{i}\right)_{1\text{\ensuremath{\le}}i\text{\ensuremath{\le}}n}\cup\left(\widetilde{e}_{p}\right)_{1\text{\ensuremath{\le}}p\le\nu}$
verifying $e_{i}\widetilde{e}_{p}=\sum_{k=1}^{n}\gamma_{ipk}e_{k}+\sum_{r=1}^{\nu}\widetilde{\gamma}_{ipr}\widetilde{e}_{r}$.
We consider the algebra $A^{o}$ with baseis $\left(a_{i}\right)_{1\leq i\leq\nu}\cup\left(\widetilde{a}_{p}\right)_{1\leq p\leq n}$
defined by $a_{i}\widetilde{a}_{p}=\sum_{k=1}^{\nu}\widetilde{\gamma}_{pik}a_{k}+\sum_{r=1}^{n}\gamma_{pir}\widetilde{a}_{r}$
then the mapping $\varphi:A\rightarrow A^{o}$ defined by $e_{i}\mapsto\widetilde{a}_{i}$
and $\widetilde{e}_{p}\mapsto a_{p}$ is an algebra-isomorphism.\end{proof}
\begin{pro}
Let $A$ be a gonosomal algebra of type $\left(n,\nu\right)$,
if $A'$ is an algebra isomorphic to $A$ then $A'$ is gonosomal
of type $\left(n,\nu\right)$ or $\left(\nu,n\right)$.\end{pro}
\begin{proof}
Let $A$ be a gonosomal algebra with basis $\left(e_{i}\right)_{1\text{\ensuremath{\le}}i\text{\ensuremath{\le}}n}\cup\left(\widetilde{e}_{p}\right)_{1\text{\ensuremath{\le}}p\le\nu}$
and $\varphi:A\rightarrow A'$ an algebra-isomorphism, we put
$a_{i}=\varphi\left(e_{i}\right)$ and $b_{p}=\varphi\left(\widetilde{e}_{p}\right)$,
we get $a_{i}a_{j}=\varphi\left(e_{i}e_{j}\right)=0$, $b_{p}b_{q}=\varphi\left(\widetilde{e}_{p}\widetilde{e}_{q}\right)=0$
and $a_{i}b_{p}=\sum_{k=1}^{n}\gamma_{ipk}a_{k}+\sum_{r=1}^{\nu}\widetilde{\gamma}_{ipr}b_{r}$,
therefore the algebra $A'$ is gonosomal for the basis $\left(a_{i}\right)_{1\text{\ensuremath{\le}}i\text{\ensuremath{\le}}n}\cup\left(b_{p}\right)_{1\text{\ensuremath{\le}}p\le\nu}$
and proposition \ref{prop:Iso-gonosomal-algebra} gives that
it can be $\left(\nu,n\right)$ type.
\end{proof}
\medskip{}

In the literature (cf. \cite{Schafer}) an algebra is referred
to as a nonassociative algebra in order to emphasize that the
associativity relation $x\left(yz\right)=\left(xy\right)z\;\left(\star\right)$
is not assumed to hold. If relation $\left(\star\right)$ is
not satisfied in an algebra, we say that this algebra is not
associative. The best-known nonassociative algebras are:
\begin{itemize}
\item Lie algebras, that is $xy+yx=0$ and $\left(xy\right)z+\left(yz\right)x+\left(zx\right)y=0$
(Jacobi identity).
\item Flexible algebras if $x\left(yx\right)=\left(xy\right)x$.
\item Alternative algebras if $x^{2}y=x\left(xy\right)$ and $yx^{2}=\left(yx\right)x$.
\item Jordan algebras if $xy=yx$ and $x^{2}\left(xy\right)=x\left(x^{2}y\right)$
(Jordan identity).
\item Power associative algebras if the subalgebra generated by any
element $x$ is associative, this is equivalent to defining $x^{1}=x$
and $x^{i+1}=xx^{i}$ and requiring $x^{i+j}=x^{i}x^{j}$ for
$i,j=1,2,\ldots$ and any $x$.
\end{itemize}
It is known that
\begin{itemize}
\item commutative algebras are flexible;
\item associative algebras are flexible, alternative, power associative
and verify the Jordan identity;
\item commutative alternative algebras are Jordan algebras;
\item Jordan algebras are power associative.
\end{itemize}
In \cite{Varro} an example of gonosomal algebra is given which
is not associative, or Lie, or alternative, or power associative,
nor Jordan. In what follows we will clarify this by showing that
gonosomal algebras constitute a new class disjoint of other nonassociative algebras.
\begin{thm}
Any gonosomal algebra is not associative, not Lie, not power
associative, not Jordan, not alternative.
\end{thm}
\begin{proof}
Let $A$ be a gonosomal algebra with basis $\left(e_{i}\right)_{1\text{\ensuremath{\le}}i\text{\ensuremath{\le}}n}\cup\left(\widetilde{e}_{j}\right)_{1\text{\ensuremath{\le}}j\text{\ensuremath{\le}}\nu}$.
For any $1\leq i,j\leq n$ and $1\leq p,q\leq\nu$ we have:
\begin{eqnarray}
e_{i}\left(e_{j}\widetilde{e}_{p}\right) & = & \sum_{k=1}^{n}\Bigl(\sum_{r=1}^{\nu}\gamma_{irk}\widetilde{\gamma}_{jpr}\Bigr)e_{k}+\sum_{s=1}^{\nu}\Bigl(\sum_{r=1}^{\nu}\widetilde{\gamma}_{irs}\widetilde{\gamma}_{jpr}\Bigr)\widetilde{e}_{s}\label{eq:ei(ejep)}\\
\left(e_{i}\widetilde{e}_{p}\right)\widetilde{e}_{q} & = & \sum_{k=1}^{n}\Bigl(\sum_{l=1}^{n}\gamma_{ipl}\gamma_{lqk}\Bigr)e_{k}+\sum_{r=1}^{\nu}\Bigl(\sum_{l=1}^{n}\gamma_{ipl}\widetilde{\gamma}_{lqr}\Bigr)\widetilde{e}_{r}.\label{eq:(eiep)eq}
\end{eqnarray}
Assuming that $A$ is associative, from $e_{i}\left(e_{j}\widetilde{e}_{p}\right)=\left(e_{i}e_{j}\right)\widetilde{e}_{p}=0$
and (\ref{eq:ei(ejep)}) we infer that
\[
\sum_{r=1}^{\nu}\gamma_{irk}\widetilde{\gamma}_{jpr}=\sum_{r=1}^{\nu}\widetilde{\gamma}_{irs}\widetilde{\gamma}_{jpr}=0,\quad\left(1\leq i,j,k\leq n,1\leq p,s\leq\nu\right)
\]
but we have
\[
\sum_{k,r=1}^{n,\nu}\gamma_{irk}\widetilde{\gamma}_{jpr}+\sum_{s,r=1}^{\nu}\widetilde{\gamma}_{irs}\widetilde{\gamma}_{jpr}=\sum_{r=1}^{\nu}\Bigl(\sum_{k=1}^{n}\gamma_{irk}+\sum_{s=1}^{\nu}\widetilde{\gamma}_{irs}\Bigr)\widetilde{\gamma}_{jpr}=\sum_{r=1}^{\nu}\widetilde{\gamma}_{jpr}
\]
and thus
\begin{equation}
\sum_{r=1}^{\nu}\widetilde{\gamma}_{jpr}=0,\quad\left(1\leq j\leq n,1\leq p\leq\nu\right).\label{eq:Sum1}
\end{equation}
Similarly, with $\left(e_{i}\widetilde{e}_{p}\right)\widetilde{e}_{q}=e_{i}\left(\widetilde{e}_{p}\widetilde{e}_{q}\right)=0$
and (\ref{eq:(eiep)eq}) we get
\[
\sum_{l=1}^{n}\gamma_{ipl}\gamma_{lqk}=\sum_{l=1}^{n}\gamma_{ipl}\widetilde{\gamma}_{lqr}=0,\quad\left(1\leq i,k\leq n,1\leq p,q,r\leq\nu\right),
\]
from which it follows that
\begin{eqnarray*}
\sum_{k,l=1}^{n}\gamma_{ipl}\gamma_{lqk}+\sum_{l,r=1}^{n,\nu}\gamma_{ipl}\widetilde{\gamma}_{lqr} & = & \sum_{l=1}^{n}\gamma_{ipl}\Bigl(\sum_{k=1}^{n}\gamma_{lqk}+\sum_{r=1}^{\nu}\widetilde{\gamma}_{lqr}\Bigr)=\sum_{l=1}^{n}\gamma_{ipl}
\end{eqnarray*}
thus
\begin{equation}
\sum_{l=1}^{n}\gamma_{ipl}=0\quad\left(1\leq i\leq n,1\leq p\leq\nu\right).\label{eq:Sum2}
\end{equation}
From relations (\ref{eq:Sum1}) and (\ref{eq:Sum2}) we get that
$\sum_{l=1}^{n}\gamma_{ipl}+\sum_{r=1}^{\nu}\widetilde{\gamma}_{ipr}=0$
for all $1\leq i\leq n,1\leq p\leq\nu$, hence a contradiction.\smallskip{}

Algebra $A$ is not a Lie algebra because if $A$ is both commutative
and anticommutative we have $xy=0$ for any $x,y\in A$, in other
words $A$ is a zero-algebra.\smallskip{}

If $A$ is a power associative algebra it verifies $x^{2}x^{2}=x^{4}$
for all $x\in A$. Let $x=e_{i}+\widetilde{e}_{p}$ where $1\leq i\leq n,1\leq p\leq\nu$,
we have:
\[
x^{2}=2\sum_{k=1}^{n}\gamma_{ipk}e_{k}+2\sum_{r=1}^{\nu}\widetilde{\gamma}_{ipr}\widetilde{e}_{r}.
\]
It follows that
\[
x^{2}x^{2}=8\sum_{l=1}^{n}\Bigl(\sum_{k,r=1}^{n,\nu}\gamma_{ipk}\widetilde{\gamma}_{ipr}\gamma_{krl}\Bigr)e_{l}+8\sum_{s=1}^{\nu}\Bigl(\sum_{k,r=1}^{n,\nu}\gamma_{ipk}\widetilde{\gamma}_{ipr}\widetilde{\gamma}_{krs}\Bigr)\widetilde{e}_{s}.
\]
but also
\[
x^{3}=2\sum_{j=1}^{n}\Theta_{j}e_{j}+2\sum_{u=1}^{\nu}\widetilde{\Theta}_{u}\widetilde{e}_{u}
\]
noting
\begin{equation}
\Theta_{j}=\sum_{k=1}^{n}\gamma_{ipk}\gamma_{kpj}+\sum_{r=1}^{\nu}\widetilde{\gamma}_{ipr}\gamma_{irj}\;\mbox{ and }\;\widetilde{\Theta}_{u}=\sum_{k=1}^{n}\gamma_{ipk}\widetilde{\gamma}_{kpu}+\sum_{r=1}^{\nu}\widetilde{\gamma}_{ipr}\widetilde{\gamma}_{iru}\label{eq:Val_Theta}
\end{equation}
and finally we get
\[
x^{4}=2\sum_{l=1}^{n}\Bigl(\sum_{j=1}^{n}\Theta_{j}\gamma_{jpl}+\sum_{u=1}^{\nu}\widetilde{\Theta}_{u}\gamma_{iul}\Bigr)e_{l}+2\sum_{s=1}^{\nu}\Bigl(\sum_{j=1}^{n}\Theta_{j}\widetilde{\gamma}_{jps}+\sum_{u=1}^{\nu}\widetilde{\Theta}_{u}\widetilde{\gamma}_{ius}\Bigr)\widetilde{e}_{s}.
\]
With the above, relation $x^{2}x^{2}=x^{4}$ implies
\begin{eqnarray*}
4\sum_{k,r=1}^{n,\nu}\gamma_{ipk}\widetilde{\gamma}_{ipr}\gamma_{krl} & = & \sum_{j=1}^{n}\Theta_{j}\gamma_{jpl}+\sum_{u=1}^{\nu}\widetilde{\Theta}_{u}\gamma_{iul}\\
4\sum_{k,r=1}^{n,\nu}\gamma_{ipk}\widetilde{\gamma}_{ipr}\widetilde{\gamma}_{krs} & = & \sum_{j=1}^{n}\Theta_{j}\widetilde{\gamma}_{jps}+\sum_{u=1}^{\nu}\widetilde{\Theta}_{u}\widetilde{\gamma}_{ius}
\end{eqnarray*}
from which it follows that
\begin{eqnarray*}
4\sum_{k,r=1}^{n,\nu}\gamma_{ipk}\widetilde{\gamma}_{ipr} & = & 4\sum_{k,r=1}^{n,\nu}\gamma_{ipk}\widetilde{\gamma}_{ipr}\Bigl(\sum_{l=1}^{n}\gamma_{krl}+\sum_{s=1}^{\nu}\widetilde{\gamma}_{krs}\Bigr)\\
 & = & \sum_{l=1}^{n}\Bigl(\sum_{j=1}^{n}\Theta_{j}\gamma_{jpl}+\sum_{u=1}^{\nu}\widetilde{\Theta}_{u}\gamma_{iul}\Bigr)+\sum_{s=1}^{\nu}\Bigl(\sum_{j=1}^{n}\Theta_{j}\widetilde{\gamma}_{jps}+\sum_{u=1}^{\nu}\widetilde{\Theta}_{u}\widetilde{\gamma}_{ius}\Bigr)\\
 & = & \sum_{j=1}^{n}\Theta_{j}\Bigl(\sum_{l=1}^{n}\gamma_{jpl}+\sum_{s=1}^{\nu}\widetilde{\gamma}_{jps}\Bigr)+\sum_{u=1}^{\nu}\widetilde{\Theta}_{u}\Bigl(\sum_{l=1}^{n}\gamma_{iul}+\sum_{s=1}^{\nu}\widetilde{\gamma}_{ius}\Bigr)\\
 & = & \sum_{j=1}^{n}\Theta_{j}+\sum_{u=1}^{\nu}\widetilde{\Theta}_{u}.
\end{eqnarray*}
But from (\ref{eq:Val_Theta}) we have:
\begin{eqnarray*}
\sum_{j=1}^{n}\Theta_{j}+\sum_{u=1}^{\nu}\widetilde{\Theta}_{u} & = & \sum_{k=1}^{n}\gamma_{ipk}\Bigl(\sum_{j=1}^{n}\gamma_{kpj}+\sum_{u=1}^{\nu}\widetilde{\gamma}_{kpu}\Bigr)+\sum_{r=1}^{\nu}\widetilde{\gamma}_{ipr}\Bigl(\sum_{j=1}^{n}\gamma_{irj}+\sum_{u=1}^{\nu}\widetilde{\gamma}_{iru}\Bigr)\\
 & = & \sum_{k=1}^{n}\gamma_{ipk}+\sum_{r=1}^{\nu}\widetilde{\gamma}_{ipr}\;=\;1
\end{eqnarray*}
thus $\Bigl(\sum_{k=1}^{n}\gamma_{ipk}\Bigr)\Bigl(\sum_{r=1}^{\nu}\widetilde{\gamma}_{ipr}\Bigr)=\frac{1}{4}$
and with $\sum_{k=1}^{n}\gamma_{ipk}+\sum_{r=1}^{\nu}\widetilde{\gamma}_{ipr}=1$
we get
\begin{equation}
\sum_{k=1}^{n}\gamma_{ipk}=\sum_{r=1}^{\nu}\widetilde{\gamma}_{ipr}=\tfrac{1}{2},\quad\left(1\leq i\leq n,1\leq p\leq\nu\right).\label{eq:Sum=00003D1/2}
\end{equation}
By linearization of $x^{2}x^{2}=x^{4}$ we get $4x^{2}\left(xy\right)=x^{3}y+x\left(x^{2}y\right)+2x\left(x\left(xy\right)\right)$
(cf. \cite{Schafer}, p. 129), we deduce that $e_{i}\left(e_{i}\left(e_{i}\widetilde{e}_{p}\right)\right)=0$.
Using (\ref{eq:ei(ejep)}) we get
\begin{eqnarray*}
e_{i}\left(e_{i}\left(e_{i}\widetilde{e}_{p}\right)\right) & = & \sum_{k=1}^{n}\Bigl(\sum_{r,s=1}^{\nu}\widetilde{\gamma}_{irs}\widetilde{\gamma}_{ipr}\gamma_{isk}\Bigr)e_{k}+\sum_{t=1}^{\nu}\Bigl(\sum_{r,s=1}^{\nu}\widetilde{\gamma}_{irs}\widetilde{\gamma}_{ipr}\widetilde{\gamma}_{ist}\Bigr)\widetilde{e}_{t}
\end{eqnarray*}
it follows that
\begin{eqnarray*}
\sum_{r,s=1}^{\nu}\widetilde{\gamma}_{irs}\widetilde{\gamma}_{ipr}\gamma_{isk} & = & \sum_{r,s=1}^{\nu}\widetilde{\gamma}_{irs}\widetilde{\gamma}_{ipr}\widetilde{\gamma}_{ist}\;=\;0,\quad\left(1\leq i,k\leq n,1\leq p,t\leq\nu\right)
\end{eqnarray*}
and therefore for all $1\leq i\leq n,1\leq p\leq\nu$ we have
\begin{eqnarray*}
\sum_{r,s=1}^{\nu}\widetilde{\gamma}_{irs}\widetilde{\gamma}_{ipr} & = & \sum_{r,s=1}^{\nu}\widetilde{\gamma}_{irs}\widetilde{\gamma}_{ipr}\Bigl(\sum_{k=1}^{n}\gamma_{isk}+\sum_{t=1}^{\nu}\widetilde{\gamma}_{ist}\Bigr)\;=\;0,
\end{eqnarray*}
But from (\ref{eq:Sum=00003D1/2}) we have:
\[
\sum_{r,s=1}^{\nu}\widetilde{\gamma}_{irs}\widetilde{\gamma}_{ipr}=\sum_{r=1}^{\nu}\widetilde{\gamma}_{ipr}\sum_{s=1}^{\nu}\widetilde{\gamma}_{irs}=\tfrac{1}{2}\sum_{r=1}^{\nu}\widetilde{\gamma}_{ipr}=\tfrac{1}{4}
\]
and so the assumtion $A$ is power associative leads to a contradiction.\end{proof}
\begin{pro}
Gonosomal algebras do not verify the Jacobi identity.\end{pro}
\begin{proof}
Let $A$ be a gonosomal algebra with basis $\left(e_{i}\right)_{1\text{\ensuremath{\le}}i\text{\ensuremath{\le}}n}\cup\left(\widetilde{e}_{j}\right)_{1\text{\ensuremath{\le}}j\text{\ensuremath{\le}}\nu}$
verifying the Jacobi identity. Applying Jacobi identity with
$\left(x,y\right)=\left(e_{i},\widetilde{e}_{p}\right)$ and
$\left(x,y\right)=\left(\widetilde{e}_{p},e_{i}\right)$ we get
$2e_{i}\left(e_{i}\widetilde{e}_{p}\right)=0$ and $2\widetilde{e}_{p}\left(\widetilde{e}_{p}e_{i}\right)=0$,
but in the previous proof to show that a gonosomal algebra is
not associative we have seen that this leads to a contradiction.
\end{proof}

\section{From gonosomal algebras to normalized gonosomal evolution operators}

Now we use Definition \ref{def:Gonosomal-Alg} with $K=\mathbb{R}$.
In this section we will associate two evolution operators with
each gonosomal $\mathbb{R}$-algebra.

\medskip{}

Starting from a gonosomal $\mathbb{R}$-algebra $A$, we define
the mapping
\begin{equation}
\begin{array}{cccc}
W: & A & \rightarrow & A\\
 & z & \mapsto & \frac{1}{2}z^{2}.
\end{array}\label{eq:Op_W_def}
\end{equation}
 In particular, if $\left(e_{i}\right)_{1\text{\ensuremath{\le}}i\text{\ensuremath{\le}}n}\cup\left(\widetilde{e}_{j}\right)_{1\text{\ensuremath{\le}}j\text{\ensuremath{\le}}\nu}$
is a gonosomal basis of $A$, for
\[
z^{\left(t\right)}=W^{t}\left(z\right)=\sum_{i=1}^{n}x_{i}^{\:\left(t\right)}e_{i}+\sum_{p=1}^{\nu}y_{p}^{\:\left(t\right)}\widetilde{e}_{p}
\]
 we find:
\begin{eqnarray}
z^{\left(t+1\right)}=W\bigl(z^{\left(t\right)}\bigr) & = & \sum_{k=1}^{n}\sum_{i,p=1}^{n,\nu}\gamma_{ipk}x_{i}^{\:\left(t\right)}y_{j}^{\:\left(t\right)}e_{k}+\sum_{r=1}^{\nu}\sum_{i,p=1}^{n,\nu}\widetilde{\gamma}_{ipr}x_{i}^{\:\left(t\right)}y_{j}^{\:\left(t\right)}\widetilde{e}_{r}.\label{eq:W(x(t))}
\end{eqnarray}
We notice that the components of the operator $W$ correspond
to the proportions obtained in (\ref{eq:Proportions}).

Note also in passing the difference between the gonosomal operator
and the evolution operator associated with an autosomal genetic
type that is defined by: $z\mapsto z^{2}$ (cf. \cite{Lyub-92},
p. 15 and \cite{WB-80}, p. 7).

For a given $z=\left(x,y\right)\in\mathbb{R}^{n}\times\mathbb{R}^{\nu}$
the dynamical system generated by $W$ is defined by the following
sequence $z$, $W\left(z\right)$, $W^{2}\left(z\right)$, $W^{3}\left(z\right)$,
\dots. Recall the quadratic evolution operator $W$ called gonosomal evolution operator  is defined in coordinate
form by:
\[
W:\mathbb{R}^{n+\nu}\rightarrow\mathbb{R}^{n+\nu},\left(x_{1},\ldots,x_{n},y_{1},\ldots,y_{n}\right)\mapsto\left(x_{1}',\ldots,x_{n}',y'_{1},\ldots,y'_{n}\right)
\]
\begin{equation}
W:\left\{ \begin{aligned}x_{k}' & =\sum_{i,j=1}^{n,\nu}\gamma_{ijk}x_{i}y_{j},\quad k=1,\ldots,n\medskip\\
y'_{r} & =\sum_{i,j=1}^{n,\nu}\widetilde{\gamma}_{ijr}x_{i}y_{j},\quad r=1,\ldots,\nu,
\end{aligned}
\right.\label{eq:Op-W}
\end{equation}
where
\begin{equation}
\sum_{k=1}^{n}\gamma_{ijk}+\sum_{r=1}^{\nu}\widetilde{\gamma}_{ijr}=1,\quad1\leq i\leq n,1\leq j\leq\nu.\label{eq:Op-W2}
\end{equation}

Conversely, it is clear that any operator of the form (\ref{eq:Op-W})
verifying (\ref{eq:Op-W2}) is associated to a gonosomal algebra.\medskip{}

An element $z^{*}\in\mathbb{R}^{n+\nu}$ is an equilibrium point
of the dynamical system (\ref{eq:Op-W}) if for all $t\geq1$
we have $W^{t}\left(z^{*}\right)=z^{*}$. It follows from the
equivalence $W^{t}\left(z^{*}\right)=z^{*},\forall t\geq1\Leftrightarrow W\left(z^{*}\right)=z^{*}$
that $z^{*}$ is an equilibrium point if and only if $z^{*}$
is a fixed point of $W$.

From the definition of $W$ we immediately deduce the following
result.
\begin{pro}
There is one-to-one correspondence between the idempotents of
the gonosomal algebra $A$ and the fixed points of the gonosomal
operator $W$associated with $A$.\end{pro}
\begin{proof}
Indeed, if $e\in A$ is an idempotent, we have $W\left(2e\right)=2e$,
i.e. $2e$ is a fixed point of $W$. And if $z^{*}\in\mathbb{R}^{n+\nu}$
is a fixed point of $W$, we get $\left(\frac{1}{2}z^{*}\right)^{2}=\frac{1}{4}\left(z^{*}\right)^{2}=\frac{1}{2}W\left(z^{*}\right)=\frac{1}{2}z^{*}$
thus element $\frac{1}{2}z^{*}$ is an idempotent of $A$.
\end{proof}
Using the definition given by (\ref{eq:Op_W_def}) we get the
following result:
\begin{pro}\label{prop:Equivalent_W} Let $\varphi:A_{1}\rightarrow A_{2}$
be an isomorphism between two gonosomal algebras $A_{1}$ and
$A_{2}$, then the gonosomal operators $W_{1}:A_{1}\rightarrow A_{1}$
and $W_{2}:A_{2}\rightarrow A_{2}$ verify $\varphi\circ W_{1}=W_{2}\circ\varphi$.\end{pro}
\begin{proof}
Indeed, for all $x\in A_{1}$ we have $\varphi\circ W_{1}\left(x\right)=\varphi\left(\frac{1}{2}x^{2}\right)=\frac{1}{2}\varphi\left(x\right)^{2}=W_{2}\circ\varphi\left(x\right)$.
\end{proof}
And this result suggests the following equivalence relation between
gonosomal operators;
\begin{defn}\label{va2}
Two gonosomal operators $W_{1}:A_{1}\rightarrow A_{1}$ and $W_{2}:A_{2}\rightarrow A_{2}$
are conjugate if and only if there exists an algebra-isomorphism
$\varphi:A_{1}\rightarrow A_{2}$ such that $\varphi\circ W_{1}=W_{2}\circ\varphi$.
\end{defn}
The trajectory of a point $z^{\left(0\right)}\in\mathbb{R}^{n+\nu}$
for the gonosomal operator $W$ is the sequence of iterations
$\bigl(z^{\left(t\right)}\bigr)_{t\geq0}$ defined by $z^{\left(t\right)}=W^{t}\bigl(z^{\left(0\right)}\bigr)$,
where each point $z^{\left(t\right)}$ corresponds to a state
of the population at generation $t$. If the trajectory of an
initial point $z^{\left(0\right)}$ converges, there is a point
$z^{\left(\infty\right)}$ such that $z^{\left(\infty\right)}=\lim_{t\rightarrow\infty}z^{\left(t\right)}$,
and by continuity of the operator $W$, the limit point $z^{\left(\infty\right)}$
is a fixed point of $W$.
\begin{pro}\label{va7}
If $W_{1}$, $W_{2}$ are two conjugate gonosomal operators,
there is an one-to-one correspondence between the fixed points
and the limit points of these two operators. \end{pro}
\begin{proof} This is very known fact see, for example \cite{De}. Here we give a brief proof. Let $\varphi:A_{1}\rightarrow A_{2}$ be the algebra-isomorphism
connecting $W_{1}$ to $W_{2}$. If $z_{1}^{*}$ is a fixed point
of $W_{1}$, by $\varphi\left(z_{1}^{*}\right)=\varphi\circ W_{1}\left(z_{1}^{*}\right)=W_{2}\circ\varphi\left(z_{1}^{*}\right)$
we get that $\varphi\left(z_{1}^{*}\right)$ is a fixed point
of $W_{2}$. And if $z_{1}^{\left(\infty\right)}$, $z_{2}^{\left(\infty\right)}$
are limit points for $W_{1}$ et $W_{2}$ respectively, we get
easily by continuity of $\varphi$: $\varphi\bigl(x_{1}^{\left(\infty\right)}\bigr)=\left(\varphi\bigl(x_{1}^{\left(0\right)}\bigr)\right)^{\left(\infty\right)}$
and $\varphi^{-1}\bigl(x_{2}^{\left(\infty\right)}\bigr)=\left(\varphi^{-1}\bigl(x_{2}^{\left(0\right)}\bigr)\right)^{\left(\infty\right)}$.\medskip{}

\end{proof}
To every gonosomal algebra $A$ is canonically attached the linear
form:
\begin{equation}
\varpi:A\rightarrow\mathbb{R},\quad\varpi\left(e_{i}\right)=\varpi\left(\widetilde{e}_{j}\right)=1.\label{eq:form_lin_pi_def}
\end{equation}
Applying $\varpi$ to (\ref{eq:W(x(t))}) we find
\begin{equation}
\varpi\bigl(z^{\left(t+1\right)}\bigr)=\sum_{i=1}^{n}x_{i}^{\:\left(t+1\right)}+\sum_{j=1}^{\nu}y_{j}^{\:\left(t+1\right)}=\bigl(\sum_{i=1}^{n}x_{i}^{\:\left(t\right)}\bigr)\bigl(\sum_{j=1}^{\nu}y_{j}^{\:\left(t\right)}\bigr)\label{omega}
\end{equation}
 which corresponds to the relation (\ref{eq:Eff-Tot}).

\medskip{}

On the fixed points of $W$ with non-negative components we have:
\begin{pro}
\label{prop:PtfixeW}If $z^{*}\in\mathbb{R}_{+}^{n+\nu}$, $z^{*}\neq0$
is a fixed point of $W$ then $\varpi\left(z^{*}\right)\geq4$.\end{pro}
\begin{proof}
Let $z^{*}=\left(x_{1},\ldots,x_{n},y_{1},\ldots,y_{\nu}\right)$
be a fixed point of $W$, with $x_{k},y_{r}\geq0$. From $W\left(z^{*}\right)=z^{*}$
we deduce that $\left(\sum_{k}x_{k}\right)\left(\sum_{r}y_{r}\right)=\sum_{k}x_{k}+\sum_{r}y_{r}=\varpi\left(z^{*}\right)$
so that $\sum_{k}x_{k}$ and $\sum_{r}y_{r}$ are positive real
roots of the polynomial $X^{2}-\varpi\left(z^{*}\right)X+\varpi\left(z^{*}\right)$
with $\varpi\left(z^{*}\right)\in\mathbb{R}_{+}$, but $\varpi\left(z^{*}\right)\left(\varpi\left(z^{*}\right)-4\right)\geq0$
and $\varpi\left(z^{*}\right)\geq0$ only if $\varpi\left(z^{*}\right)\geq4$.
\end{proof}
\medskip{}

For applications in genetics we restrict to the simplex of $\mathbb{R}^{n+\nu}$:
\[
S^{\:n+\nu-1}=\left\{ \left(x_{1},\ldots,x_{n},y_{1},\ldots,y_{\nu}\right)\in\mathbb{R}^{n+\nu}:x_{i}\geq0,y_{i}\geq0,\sum_{i=1}^{n}x_{i}+\sum_{i=1}^{\nu}y_{i}=1\right\}
\]
this simplex is associated with frequency distributions of the
genetic types $e_{i}$ and $\widetilde{e}_{j}$. But the gonosomal
operator $W$ does not preserve the simplex $S^{\:n+\nu-1}$,
indeed :
\begin{pro}\label{pro}
Let $A$ be a gonosomal $\mathbb{R}$-algebra of type $\left(n,\nu\right)$,
we have:

a) $W\left(\mathbb{R}_{+}^{n+\nu}\right)\subset\mathbb{R}_{+}^{n+\nu}$
if and only if $\gamma_{ijk}\geq0$ and $\widetilde{\gamma}_{ijr}\geq0$
for all $1\leq i,k\leq n$ and $1\leq j,r\leq\nu$.

b) $\varpi\circ W\left(z\right)\leq\frac{1}{4}$ for all $z\in S^{\:n+\nu-1}$.\end{pro}
\begin{proof}
For \emph{a}) the sufficient condition is immediate. For the
necessary condition it suffices to note that $W\left(e_{i}+\widetilde{e}_{j}\right)=\sum_{k=1}^{n}\gamma_{ijk}e_{k}+\sum_{k=1}^{n}\widetilde{\gamma}_{ijk}\widetilde{e}_{k}$
for every $1\leq i\leq n$ and $1\leq j\leq\nu$.\emph{ }Result
\emph{b}) follows from the well known inequality $4ab\leq\left(a+b\right)^{2}$.
\end{proof}
This leads to the following definition.
\begin{defn}
We say that a $K$-algebra $A$ is a gonosomal stochastic algebra
of type $\left(n,\nu\right)$ if it satisfies the definition
\ref{def:Gonosomal-Alg} with $K=\mathbb{R}$ and $\gamma_{ipk}\geq0$,
$\widetilde{\gamma}_{ipr}\geq0$ for all $1\leq i,k\leq n$ and
$1\leq p,r\leq\nu$.
\end{defn}
In a gonosomal stochastic algebra with basis $\left(e_{i}\right)_{1\text{\ensuremath{\le}}i\text{\ensuremath{\le}}n}\cup\left(\widetilde{e}_{p}\right)_{1\text{\ensuremath{\le}}p\text{\ensuremath{\le}}\nu}$,
the elements of $\left(e_{i}\right)_{1\text{\ensuremath{\le}}i\text{\ensuremath{\le}}n}$
(resp. $\left(\widetilde{e}_{p}\right)_{1\text{\ensuremath{\le}}p\text{\ensuremath{\le}}\nu}$)
represent genetic types observed in females (resp. in males),
and the structure constants $\gamma_{ipk}$ (resp. $\widetilde{\gamma}_{ipr}$)
are the inheritance coefficients, that is to say the probability
that a female (resp. a male) offspring is of type $e_{k}$ (resp.
$\widetilde{e}_{r}$) when the parental pair is a female of type
$e_{i}$ and a male of type $\widetilde{e}_{p}$.
\begin{pro}
Let $A$ be a gonosomal stochastic algebra of type $\left(n,\nu\right)$
and $z\in\mathbb{R}_{+}^{n+\nu}$.

\qquad{}a) If $\varpi\left(z\right)=0$ then $z=0$.

For all $t\geq1$ we denote $z^{\left(t\right)}=W^{t}\bigl(z\bigr)$,
then we have:

\qquad{}b) If $\varpi\bigl(z\bigr)\leq4$ , the sequence $\left(\varpi\bigl(z^{\left(t\right)}\bigr)\right)_{t\geq0}$
is decreasing.

\qquad{}c) For $t\geq0$,
\[
\Bigl(\min_{i,j}\Bigl\{\sqrt{\gamma_{ij}\widetilde{\gamma}_{ij}}\Bigr\}\Bigr)^{2\left(2^{t}-1\right)}\left(\varpi\bigl(z\bigr)\right)^{2^{t}}\leq\varpi\bigl(z^{\left(t\right)}\bigr)\leq\Bigl(\max_{i,j,p,q}\left\{ \gamma_{ij}\widetilde{\gamma}_{pq}\right\} \Bigr)^{2^{t}-1}\left(\varpi\bigl(z\bigr)\right)^{2^{t}},
\]

\[
\varpi\bigl(z^{\left(t\right)}\bigr)\leq\left(\max_{i,j,p,q}\left\{ \tfrac{1}{16}\gamma_{ij}\widetilde{\gamma}_{pq}\right\} \right)^{\frac{1}{3}\bigl(4^{\left\lfloor \nicefrac{t}{2}\right\rfloor }-1\bigr)}\times\begin{cases}
\;\;\Bigl(\varpi\bigl(z\bigr)\Bigr){}^{4^{\left\lfloor \nicefrac{t}{2}\right\rfloor }} & \mbox{if }t\mbox{ is even,}\medskip\\
\Bigl(\tfrac{1}{4}\varpi\bigl(z\bigr)\Bigr){}^{4^{\left\lfloor \nicefrac{t}{2}\right\rfloor }} & \mbox{if }t\mbox{ is odd},
\end{cases}
\]
where we put $\gamma_{ij}=\sum_{k=1}^{n}\gamma_{ijk}$ and $\widetilde{\gamma}_{pq}=\sum_{r=1}^{\nu}\widetilde{\gamma}_{pqr}$
for all $1\leq i,p\leq n$ and $1\leq j,q\leq\nu$.\end{pro}
\begin{proof}
\emph{a})\emph{ }Immediate.

In what follows for all $t\geq0$ we note $z^{\left(t\right)}=\left(x_{1}^{\:\left(t\right)},\ldots,x_{n}^{\:\left(t\right)},y_{1}^{\:\left(t\right)},\ldots,y_{\nu}^{\:\left(t\right)}\right)$
where $z^{\left(0\right)}=z$.

\emph{b}) We show recursively with the relations (\ref{eq:Op-W})
that $z^{\left(t\right)}\in\mathbb{R}_{+}^{n+\nu}$ for every
$t\geq0$. From
\[
4\Bigl(\sum_{k=1}^{n}x_{k}^{\:\left(t-1\right)}\Bigr)\Bigl(\sum_{r=1}^{\nu}y_{r}^{\:\left(t-1\right)}\Bigr)\leq\Bigl(\sum_{k=1}^{n}x_{k}^{\:\left(t-1\right)}+\sum_{r=1}^{\nu}y_{r}^{\:\left(t-1\right)}\Bigr)^{2}
\]
we deduce that we have for all $t\geq1$ :
\[
4\varpi\bigl(z^{\left(t\right)}\bigr)\leq\Bigl(\varpi\bigl(z^{\left(t-1\right)}\bigr)\Bigr)^{2},\qquad\left(*\right)
\]
from $0\leq\varpi\left(z\right)\leq4$ we infer that $\Bigl(\varpi\bigl(z\bigr)\Bigr)^{2}\leq4\varpi\bigl(z\bigr)$
and with $\left(*\right)$ it follows $\varpi\bigl(z^{\left(1\right)}\bigr)\leq\varpi\bigl(z\bigr)\leq4$
then by $\left(*\right)$ and by induction the result is obtained.

\emph{c}) Indeed, from (\ref{omega}) we have:
\begin{eqnarray*}
\varpi\bigl(z^{\left(t\right)}\bigr) & = & \Bigl(\sum_{k=1}^{n}x_{k}^{\:\left(t-1\right)}\Bigr)\Bigl(\sum_{r=1}^{\nu}y_{r}^{\:\left(t-1\right)}\Bigr)
\end{eqnarray*}
with relations (\ref{eq:Op-W}) this is written
\begin{eqnarray}
\varpi\bigl(z^{\left(t\right)}\bigr) & = & \Bigl(\sum_{i,j=1}^{n,\nu}\gamma_{ij}\:x_{i}^{\:\left(t-2\right)}y_{j}^{\;\left(t-2\right)}\Bigr)\Bigl(\sum_{p,q=1}^{n,\nu}\widetilde{\gamma}_{pq}\:x_{p}^{\:\left(t-2\right)}y_{q}^{\;\left(t-2\right)}\Bigr)\nonumber \\
 & = & \sum_{i,p=1}^{n}\sum_{j,q=1}^{\nu}\gamma_{ij}\widetilde{\gamma}_{pq}\:x_{i}^{\:\left(t-2\right)}x_{p}^{\:\left(t-2\right)}y_{j}^{\;\left(t-2\right)}y_{q}^{\;\left(t-2\right)}\label{eq:Rel1}
\end{eqnarray}
consequently
\begin{equation}
\varpi\bigl(z^{\left(t\right)}\bigr)\leq\max_{i,j,p,q}\left\{ \gamma_{ij}\widetilde{\gamma}_{pq}\right\} \Bigl(\sum_{i=1}^{n}x_{i}^{\:\left(t-2\right)}\Bigr)^{2}\Bigl(\sum_{j=1}^{\nu}y_{j}^{\;\left(t-2\right)}\Bigr)^{2}\label{eq:Rel2}
\end{equation}
but from (\ref{omega}) we have $\Bigl(\sum_{k=1}^{n}x_{k}^{\:\left(t-2\right)}\Bigr)\Bigl(\sum_{r=1}^{\nu}y_{r}^{\;\left(t-2\right)}\Bigr)=\varpi\left(z^{\left(t-1\right)}\right)$
and thus
\begin{eqnarray*}
\varpi\bigl(z^{\left(t\right)}\bigr) & \leq & \max_{i,j,p,q}\left\{ \gamma_{ij}\widetilde{\gamma}_{pq}\right\} \left(\varpi\bigl(z^{\left(t-1\right)}\bigr)\right)^{2},
\end{eqnarray*}
we deduce by induction: $\varpi\bigl(z^{\left(t\right)}\bigr)\leq\Bigl(\max_{i,j,p,q}\left\{ \gamma_{ij}\widetilde{\gamma}_{pq}\right\} \Bigr)^{2^{t}-1}\left(\varpi\bigl(z\bigr)\right)^{2^{t}}$.

By exchanging the roles of $\left(i,j\right)$ and $\left(p,q\right)$
in (\ref{eq:Rel1}) we obtain:
\begin{eqnarray*}
\varpi\bigl(z^{\left(t\right)}\bigr) & = & \sum_{i,p=1}^{n}\sum_{j,q=1}^{\nu}\gamma_{pq}\widetilde{\gamma}_{ij}\:x_{i}^{\:\left(t-2\right)}x_{p}^{\:\left(t-2\right)}y_{j}^{\;\left(t-2\right)}y_{q}^{\;\left(t-2\right)}
\end{eqnarray*}
hence
\begin{eqnarray*}
\varpi\bigl(z^{\left(t\right)}\bigr) & = & \sum_{i,p=1}^{n}\sum_{j,q=1}^{\nu}\tfrac{1}{2}\left(\gamma_{ij}\widetilde{\gamma}_{pq}+\gamma_{pq}\widetilde{\gamma}_{ij}\right)\:x_{i}^{\:\left(t-2\right)}x_{p}^{\:\left(t-2\right)}y_{j}^{\;\left(t-2\right)}y_{q}^{\;\left(t-2\right)}
\end{eqnarray*}
but from $a+b\geq2\sqrt{ab}$ it follows
\begin{eqnarray*}
\varpi\bigl(z^{\left(t\right)}\bigr) & \geq & \sum_{i,p=1}^{n}\sum_{j,q=1}^{\nu}\sqrt{\gamma_{ij}\gamma_{pq}\widetilde{\gamma}_{ij}\widetilde{\gamma}_{pq}}\:x_{i}^{\:\left(t-2\right)}x_{p}^{\:\left(t-2\right)}y_{j}^{\;\left(t-2\right)}y_{q}^{\;\left(t-2\right)}\\
 & = & \Bigl(\sum_{i,j=1}^{n,\nu}\sqrt{\gamma_{ij}\widetilde{\gamma}_{ij}}\:x_{i}^{\:\left(t-2\right)}y_{j}^{\;\left(t-2\right)}\Bigr)^{2}\\
 & \geq & \Bigl(\min_{i,j}\Bigl\{\sqrt{\gamma_{ij}\widetilde{\gamma}_{ij}}\Bigr\}\Bigr)^{2}\Bigl(\sum_{i=1}^{n}x_{i}^{\:\left(t-2\right)}\Bigr)^{2}\Bigl(\sum_{j=1}^{\nu}y_{j}^{\;\left(t-2\right)}\Bigr)^{2}
\end{eqnarray*}
consequently
\begin{eqnarray*}
\Bigl(\min_{i,j}\Bigl\{\sqrt{\gamma_{ij}\widetilde{\gamma}_{ij}}\Bigr\}\Bigr)^{2}\Bigl(\varpi\bigl(z^{\left(t-1\right)}\bigr)\Bigr)^{2} & \leq & \varpi\bigl(z^{\left(t\right)}\bigr),
\end{eqnarray*}
and we deduce by induction that $\Bigl(\min_{i,j}\Bigl\{\sqrt{\gamma_{ij}\widetilde{\gamma}_{ij}}\Bigr\}\Bigr)^{2\left(2^{t}-1\right)}\left(\varpi\bigl(z\bigr)\right)^{2^{t}}\leq\varpi\bigl(z^{\left(t\right)}\bigr)$.

From (\ref{eq:Rel2}) using (\ref{omega}) and $ab\leq\frac{1}{4}\left(a+b\right)^{2}$
it follows that
\[
\varpi\bigl(z^{\left(t\right)}\bigr)\leq\max_{i,j,p,q}\left\{ \tfrac{1}{16}\gamma_{ij}\widetilde{\gamma}_{pq}\right\} \Bigl(\varpi\left(z^{\left(t-2\right)}\right)\Bigr)^{4}
\]
thus by induction
\[
\varpi\bigl(z^{\left(t\right)}\bigr)\leq\Bigl(\max_{i,j,p,q}\left\{ \tfrac{1}{16}\gamma_{ij}\widetilde{\gamma}_{pq}\right\} \Bigr)^{\frac{1}{3}\left(4^{\left\lfloor \nicefrac{t}{2}\right\rfloor }-1\right)}\Bigl(\varpi\left(z^{\left(t-2\left\lfloor \frac{t}{2}\right\rfloor \right)}\right)\Bigr)^{4^{\left\lfloor \nicefrac{t}{2}\right\rfloor }}
\]
we deduce immediately the result when $t$ is even and when $t$
is odd it suffices to note that $\varpi\bigl(z^{\left(1\right)}\bigr)=\left(\sum_{k}x_{k}\right)\left(\sum_{r}y_{r}\right)\leq\frac{1}{4}\left(\varpi\bigl(z\bigr)\right)^{2}$.
\end{proof}
Denote
\[
\mathcal{O}^{\:n,\nu}=\left\{ \left(x_{1},\ldots,x_{n},y_{1},\ldots,y_{\nu}\right)\in\mathbb{R}^{n+\nu}:x_{1}=\cdots=x_{n}=0\mbox{ or }y_{1}=\cdots=y_{\nu}=0\right\}.
\]
It is easy to see that for $z\in\mathbb{R}_{+}^{n+\nu}$ we have:
\[
\varpi\circ W\left(z\right)=\Bigl(\sum_{i=1}^{n}x_{i}\Bigr)\Bigl(\sum_{j=1}^{\nu}y_{j}\Bigr)=0\;\Leftrightarrow\;z\in\mathcal{O}^{\:n,\nu}.
\]

Therefore if we denote
\[
S^{\:n,\nu}=S^{\:n+\nu-1}\setminus\mathcal{O}^{\:n,\nu}
\]
then the operator
\[
V:S^{\:n,\nu}\rightarrow S^{\:n,\nu},\quad z\mapsto\frac{1}{\varpi\circ W\left(z\right)}W\left(z\right)
\]
is well defined, it is called the normalized gonosomal operator
of $W$. Using the relations (\ref{eq:Op-W}) we can express
the operator $V$ in coordinate form by:
\begin{equation}
V:\left\{ \begin{aligned}x_{k}' & =\frac{\sum_{i,j=1}^{n,\nu}\gamma_{ijk}x_{i}y_{j}}{\left(\sum_{i=1}^{n}x_{i}\right)\bigl(\sum_{j=1}^{\nu}y_{j}\bigr)},\quad k=1,\ldots,n\medskip\\
y'_{r} & =\frac{\sum_{i,j=1}^{n,\nu}\widetilde{\gamma}_{ijr}x_{i}y_{j}}{\left(\sum_{i=1}^{n}x_{i}\right)\bigl(\sum_{j=1}^{\nu}y_{j}\bigr)},\quad r=1,\ldots,\nu.
\end{aligned}
\right.\label{eq:Op-V}
\end{equation}
We can notice that the coordinates of the operator $V$ correspond
to the frequency distributions of genetic types obtained in (\ref{eq:Frequencies}).

\medskip{}

\begin{pro}
Let $A$ be a gonosomal stochastic algebra of type $\left(n,\nu\right)$.
For all $z\in S^{\:n,\nu}$ and $t\geq1$ we define $z^{\left(t\right)}=V^{t}\bigl(z\bigr)=\left(x_{1}^{\:\left(t\right)},\ldots,x_{n}^{\:\left(t\right)},y_{1}^{\;\left(t\right)},\ldots,y_{\nu}^{\;\left(t\right)}\right)$,
then we have
\[
\min_{i,j}\left\{ \gamma_{ijk}\right\} \leq x_{k}^{\:\left(t\right)}\leq\max_{i,j}\left\{ \gamma_{ijk}\right\} \quad\mbox{and}\quad\min_{i,j}\left\{ \widetilde{\gamma}_{ijr}\right\} \leq y_{r}^{\;\left(t\right)}\leq\max_{i,j}\left\{ \widetilde{\gamma}_{ijr}\right\} .
\]
\end{pro}
\begin{proof}
It is easy to see that for each $1\leq k\leq n$ and $1\leq r\leq\nu$ the following inequalities hold
\[
\min_{i,j}\left\{ \gamma_{ijk}\right\} \Bigl(\sum_{i,j}x_{i}^{\:\left(t-1\right)}y_{j}^{\;\left(t-1\right)}\Bigr)\leq\sum_{i,j}\gamma_{ijk}x_{i}^{\:\left(t-1\right)}y_{j}^{\;\left(t-1\right)}\leq\max_{i,j}\left\{ \gamma_{ijk}\right\} \Bigl(\sum_{i,j}x_{i}^{\:\left(t-1\right)}y_{j}^{\;\left(t-1\right)}\Bigr)
\]
\[
\min_{i,j}\left\{ \widetilde{\gamma}_{ijr}\right\} \Bigl(\sum_{i,j}x_{i}^{\:\left(t-1\right)}y_{j}^{\;\left(t-1\right)}\Bigr)\leq\sum_{i,j}\widetilde{\gamma}_{ijr}x_{i}^{\:\left(t-1\right)}y_{j}^{\;\left(t-1\right)}\leq\max_{i,j}\left\{ \widetilde{\gamma}_{ijr}\right\} \Bigl(\sum_{i,j}x_{i}^{\:\left(t-1\right)}y_{j}^{\;\left(t-1\right)}\Bigr),
\]
therefore the result follows using relations (\ref{eq:Op-V}).
\end{proof}
We can study the action of an algebra-isomorphism on normalized
gonosomal operators.
\begin{pro}
If $A_{1}$ and $A_{2}$ are gonosomal stochastic algebras, $\varpi_{1}$
and $\varpi_{2}$ the linear forms defined on $A_{1}$ and $A_{2}$
as in (\ref{eq:form_lin_pi_def}) and if $\varphi:A_{1}\rightarrow A_{2}$
is an algebra-isomorphism such that $\varpi_{2}\circ\varphi=\varpi_{1}$
then we have $V_{2}=\varphi\circ V_{1}\circ\varphi^{-1}$. \end{pro}
\begin{proof}
According to Proposition \ref{prop:Equivalent_W} we have
$\varphi\circ W_{1}=W_{2}\circ\varphi$. It is easy to show that
for $z\in\mathbb{R}^{n+\nu}$ we get: $\varpi_{1}\circ W_{1}\left(z\right)=0\;\Leftrightarrow\;\varpi_{2}\circ W_{2}\left(z\right)=0$.
And for all $z\in S^{\:n,\nu}$ we get:
\begin{eqnarray*}
V_{2}\circ\varphi\left(z\right) & = & \frac{1}{\varpi_{2}\circ W_{2}\circ\varphi\left(z\right)}W_{2}\circ\varphi\left(z\right)=\frac{1}{\varpi_{2}\circ\varphi\circ W_{1}\left(z\right)}\varphi\circ W_{1}\left(z\right)\\
 & = & \frac{1}{\varpi_{1}\circ W_{1}\left(z\right)}\varphi\circ W_{1}\left(z\right)=\varphi\circ V_{1}\left(z\right).
\end{eqnarray*}
\end{proof}
\begin{pro}
\label{prop:W=000026O}In a gonosomal stochastic algebra of type
$\left(n,\nu\right)$:

a) If there is $t_{0}\geq1$ such that $W^{t_{0}}\bigl(z\bigr)=0$
then $W^{t}\bigl(z\bigr)=0$ for all $t\geq t_{0}$.

b) If there is $t\geq0$ such that $W^{t}\left(z\right)\in\mathcal{O}^{\:n,\nu}$
then $W^{t+1}\left(z\right)=0$.

c) For $z\in\mathbb{R}_{+}^{n+\nu}$ and $t\geq0$ we have $W^{t}\left(z\right)\in\mathcal{O}^{\:n,\nu}\;\Leftrightarrow\;\varpi\circ W^{t+1}\left(z\right)=0$.

d) For $z\in\mathbb{R}_{+}^{n+\nu}$, $z\neq0$, if $W^{t}\left(z\right)=0$
then there is $0\leq t_{0}<t$ such that $W^{t_{0}}\left(z\right)\neq0$
and $W^{t_{0}}\left(z\right)\in\mathcal{O}^{\:n,\nu}$.

e) For all $z\in S^{\:n,\nu}$ and $t\geq0$ such that $\varpi\circ W^{t}\left(z\right)\neq0$
we have:
\[
V^{t}\left(z\right)=\frac{1}{\varpi\circ W^{t}\left(z\right)}W^{t}\left(z\right).
\]
\end{pro}
\begin{proof}
\emph{a}) With $z^{\:\left(t\right)}=\left(x_{1}^{\:\left(t\right)},\ldots,x_{n}^{\:\left(t\right)},y_{1}^{\:\left(t\right)},\ldots,y_{n}^{\:\left(t\right)}\right)$,
from $W^{t_{0}}\bigl(z\bigr)=0$ we have $x_{i}^{\:\left(t_{0}\right)}=0$
and $y_{j}^{\:\left(t_{0}\right)}=0$ what implies according
to (\ref{eq:Op-W}): $x_{i}^{\:\left(t_{0}+1\right)}=0$ and
$y_{j}^{\:\left(t_{0}+1\right)}=0$ and the result follows by
induction.

\emph{b}) For  $W^{t}\left(z\right)=\left(x_{1},\ldots,x_{n},y_{1},\ldots,y_{\nu}\right)$,
if $x_{k}=0$ for all $1\leq k\leq n$ or $y_{r}=0$ and
$1\leq r\leq\nu$ then from relations (\ref{eq:Op-W}) we get
$x'_{k}=0$ and $y'_{r}=0$ and thus $W^{t+1}\left(z\right)=0$.

\emph{c}) Necessity follows from \emph{b}). For the sufficiency,
it is enough to see that $W^{t}\left(z\right)=\left(x_{1},\ldots,x_{n},y_{1},\ldots,y_{\nu}\right)$
implies $\varpi\circ W^{t+1}\left(z\right)=\left(\sum_{k=1}^{n}x_{k}\right)\left(\sum_{r=1}^{\nu}y_{r}\right)$,
therefore if $\varpi\circ W^{t+1}\left(z\right)=0$ then we get $\sum_{k=1}^{n}x_{k}=0$
or $\sum_{r=1}^{\nu}y_{r}=0$ and as $x_{k}\geq0$, $y_{r}\geq0$
for all $k$ and $r$ we have $W^{t}\left(z\right)\in\mathcal{O}^{\:n,\nu}$.

\emph{d}) Let $z\neq0$ and $t>0$. Let $t_{0}\geq0$
be the smallest integer such that $W^{t_{0}+1}\left(z\right)=0$,
thus $t_{0}+1\leq t$, from $\varpi\circ W^{t_{0}+1}\left(z\right)=0$
and \emph{c}) we deduce that $W^{t_{0}}\left(z\right)\in\mathcal{O}^{\:n,\nu}$.

\emph{e}) By induction on $t\geq0$. For $t\geq1$, suppose
that $\varpi\circ W^{t+1}\left(z\right)\neq0$ and that $V^{t}\left(z\right)=\frac{1}{\varpi\circ W^{t}\left(x\right)}W^{t}\left(z\right)$
then we have $W\Bigl(V^{t}\bigl(z\bigr)\Bigr)=\left(\frac{1}{\varpi\circ W^{t}\left(z\right)}\right)^{2}W^{t+1}\left(z\right)\;\left(*\right)$
from which it follows $\varpi\circ W\Bigl(V^{t}\bigl(z\bigr)\Bigr)=\left(\frac{1}{\varpi\circ W^{t}\left(z\right)}\right)^{2}\varpi\circ W^{t+1}\left(z\right)\neq0\;\left(**\right)$.
By definition of the operator $V$ we get
\[
V^{t+1}\left(z\right)=V\left(V^{t}\left(z\right)\right)=\frac{1}{\varpi\circ W\left(V^{t}\left(z\right)\right)}W\left(V^{t}\left(z\right)\right)
\]
what with $\left(*\right)$ and $\left(**\right)$ gives the
relation to the order $t+1$.\end{proof}
\begin{rk}
From a genetic point of view, the result \emph{a}) means that
in a bisexual population when a sex-linked gonosomal gene disappears
it does not reappear. Results \emph{b}) and \emph{c}) means that
all individuals of one sex disappear if and only if a gonosomal
gene disappears.
\end{rk}
\medskip{}

There is a relation between the fixed points of the operator
$V$ and some fixed points of $W$, for this we introduce the
following definition: a fixed point $z=\left(x_{1},\ldots,x_{n},y_{1},\ldots,y_{\nu}\right)$
of the gonosomal operator $W$ is non-negative and normalizable
if it satisfies the following conditions $x_{i},y_{j}\geq0$
and $\sum_{i=1}^{n}x_{i}+\sum_{j=1}^{\nu}y_{j}>0$. It has been
shown in \cite{Roz-Va} that
\begin{pro}
The map $z^{*}\mapsto\frac{1}{\varpi\left(z^{*}\right)}z^{*}$
is an one-to-one correspondence between the set of non-negative
and normalizable fixed point of $W$ and the set of fixed points
of the operator $V$.
\end{pro}
The various stability notions of the equilibrium points are preserved
by passing from $W$ to the operator $V$.
\begin{thm}
\label{prop:Notions_Equilibrium_pts}Let $z^{*}$ be a non-negative
and normalizable fixed point of $W$.

a) If $z^{*}$ is attractive then $\frac{1}{\varpi\left(z^{*}\right)}z^{*}$
is an attractive equilibrium point of $V$.

b) If $z^{*}$ is stable (resp. uniformly stable) then $\frac{1}{\varpi\left(z^{*}\right)}z^{*}$
is a stable (resp. uniformly stable) equilibrium point of $V$.

c) If $z^{*}$ is asymptotically stable then the fixed point
$\frac{1}{\varpi\left(z^{*}\right)}z^{*}$ of $V$ is asymptotically
stable.

d) If $z^{*}$ is exponentially stable then the fixed point $\frac{1}{\varpi\left(z^{*}\right)}z^{*}$
of $V$ is exponentially stable.\end{thm}
\begin{proof}
\emph{a}) If $z^{*}$ is an attractive point of $W$, then there is
$\rho>0$ such that for all $z\in\mathbb{R}^{n+\nu}$ verifying
$\left\Vert z-z^{*}\right\Vert <\rho$ we have $\lim_{t\rightarrow\infty}W^{t}\left(z\right)=z^{*}$.
As $z^{*}\neq0$ we get $\varpi\left(z^{*}\right)\neq0$. By
continuity of $\varpi$ we have $\lim_{t\rightarrow\infty}\varpi\circ W^{t}\left(z\right)=\varpi\left(z^{*}\right)$.
Next for all $z\in\mathbb{R}^{n+\nu}$ such that $\lim_{t\rightarrow\infty}W^{t}\left(z\right)=z^{*}$
we get $W^{t}\left(z\right)\neq0$ for every $t\geq0$, otherwise
according to Proposition \ref{prop:W=000026O} \emph{a}), we
would have $\lim_{t\rightarrow\infty}W^{t}\left(z\right)=0$,
we deduce that, in particular if $z\in S^{\:n+\nu-1}$ we get
$\varpi\circ W^{t}\left(z\right)\neq0$. Finally, for any $z\in S^{\:n+\nu-1}$
such that $\left\Vert z-z^{*}\right\Vert <\rho$ we get $\lim_{t\rightarrow\infty}V^{t}\left(z\right)=\lim_{t\rightarrow\infty}\frac{1}{\varpi\circ W^{t}\left(z\right)}W^{t}\left(z\right)=\frac{1}{\varpi\left(z^{*}\right)}z^{*}$.\smallskip{}

In the following $\mathbb{R}^{n+\nu}$ is equipped with the norm
$\left\Vert \left(x_{1},\ldots,x_{n+\nu}\right)\right\Vert =\sum_{i=1}^{n+\nu}\left|x_{i}\right|$
and we see that for this norm we have $\left\Vert z\right\Vert =\varpi\left(z\right)$
if $z\in\mathbb{R}_{+}^{n+\nu}$.\smallskip{}

\emph{b}) By definition, the equilibrium point $z^{*}$ is stable
for $W$ if for all $t_{0}\geq0$ and $\epsilon>0$, there exists
$\delta>0$ such that the condition $\left\Vert z-z^{*}\right\Vert <\delta$
implies $\left\Vert W^{t}\left(z\right)-z^{*}\right\Vert <\epsilon\;\left(t\geq t_{0}\right)$,
and $z^{*}$ is uniformy stable if the existence of $\delta>0$
does not depend on $t_{0}$.

We deduce from Proposition \ref{prop:PtfixeW} that $\varpi\left(z^{*}\right)-2>2$,
in what follows we take $0<\epsilon<\varpi\left(z^{*}\right)-2$.
For all $z\in S^{\:n,\nu}$ we get
\[
\left\Vert V^{t}\left(z\right)-V\left(z^{*}\right)\right\Vert \leq\left\Vert \tfrac{1}{\varpi\circ W^{t}\left(z\right)}W^{t}\left(z\right)-\tfrac{1}{\varpi\circ W^{t}\left(z\right)}z^{*}\right\Vert +\left\Vert \tfrac{1}{\varpi\circ W^{t}\left(z\right)}z^{*}-\tfrac{1}{\varpi\left(z^{*}\right)}z^{*}\right\Vert
\]
or
\begin{equation}
\left\Vert V^{t}\left(z\right)-V\left(z^{*}\right)\right\Vert \leq\tfrac{1}{\varpi\circ W^{t}\left(z\right)}\left\Vert W^{t}\left(z\right)-z^{*}\right\Vert +\left|\tfrac{1}{\varpi\circ W^{t}\left(z\right)}-\tfrac{1}{\varpi\left(z^{*}\right)}\right|\left\Vert z^{*}\right\Vert .\label{eq:inegal1}
\end{equation}
If we denote $W^{t}\left(z\right)=\bigl(x_{i}^{\:\left(t\right)}\bigr)_{1\leq i\leq n+\nu}$
and $z^{*}=\left(x_{i}^{*}\right)_{1\leq i\leq n+\nu}$ we notice
that
\[
\bigl|\varpi\circ W^{t}\left(z\right)-\varpi\left(z^{*}\right)\bigr|\leq\sum_{i=1}^{n+\nu}\bigl|x_{i}^{\:\left(t\right)}-x_{i}^{*}\bigr|=\left\Vert W^{t}\left(z\right)-z^{*}\right\Vert ,
\]
we deduce that for all $x\in S^{\:n,\nu}$ such that $\left\Vert z-z^{*}\right\Vert <\delta$
we have $0<\varpi\left(z^{*}\right)-\epsilon\leq\varpi\circ W^{t}\left(z\right)$,
with this and $\left\Vert z^{*}\right\Vert =\varpi\left(z^{*}\right)$
inequality (\ref{eq:inegal1}) becomes
\[
\left\Vert V^{t}\left(z\right)-V\left(z^{*}\right)\right\Vert \leq\tfrac{2\epsilon}{\varpi\left(z^{*}\right)-\epsilon}<\epsilon
\]
which proves the result.

\emph{c}) If $z^{*}$ is asymptotically stable for $W$, then
by definition $z^{*}$ is attractive and stable for $W$ but
from \emph{a}) and \emph{b}) it follows that $z^{*}$ is attractive
and stable for $V$, thus $z^{*}$ is asymptotically stable for
$V$.

\emph{d}) By definition, the equilibrium point $z^{*}$ of $W$
is exponentially stable if for all $t_{0}\geq0$ there exists
$\delta>0$, $M>0$ and $\eta\in\left]0,1\right[$ such that
for $z\in\mathbb{R}^{n+\nu}$ :
\[
\left\Vert z-z^{*}\right\Vert \leq\delta\Rightarrow\left\Vert W^{t}\left(z\right)-z^{*}\right\Vert \leq M\eta^{t}\left\Vert z-z^{*}\right\Vert ,\;\mbox{for all }t\geq t_{0}.
\]
 Analogously to what was done in \emph{b}), for all $z\in S^{\:n,\nu}$
we have the inequality:
\begin{equation}
\left\Vert V^{t}\left(z\right)-V\left(z^{*}\right)\right\Vert \leq\tfrac{1}{\varpi\circ W^{t}\left(z\right)}\left\Vert W^{t}\left(z\right)-z^{*}\right\Vert +\left|\tfrac{1}{\varpi\circ W^{t}\left(z\right)}-\tfrac{1}{\varpi\left(z^{*}\right)}\right|\left\Vert z^{*}\right\Vert .\label{eq:inegal2}
\end{equation}

As in \emph{b}) we get: $\bigl|\varpi\circ W^{t}\left(z\right)-\varpi\left(z^{*}\right)\bigr|\leq\left\Vert W^{t}\left(z\right)-z^{*}\right\Vert $,
we deduce that for all $z\in S^{\:n,\nu}$ verifying $\left\Vert z-z^{*}\right\Vert \leq\delta$
we get $\varpi\left(z^{*}\right)-M\eta^{t}\left\Vert z-z^{*}\right\Vert \leq\varpi\circ W^{t}\left(z\right)$.
But $\eta\in\left]0,1\right[$, thus there exists $t_{1}\geq t_{0}$
such that $4-M\eta^{t}\left\Vert z-z^{*}\right\Vert \geq2$ for
$t\geq t_{1}$, but we saw in Proposition \ref{prop:PtfixeW}
that $\varpi\left(z^{*}\right)\geq4$, thus for all $z\in S^{\:n,\nu}$
such that $\left\Vert z-z^{*}\right\Vert \leq\delta$ and for
every $t\geq t_{1}$ we have
\[
2\leq\varpi\left(z^{*}\right)-M\eta^{t}\left\Vert z-z^{*}\right\Vert \leq\varpi\circ W^{t}\left(z\right)
\]
with this and $\left\Vert z^{*}\right\Vert =\varpi\left(z^{*}\right)$,
inequality (\ref{eq:inegal2}) becomes
\[
\left\Vert V^{t}\left(z\right)-V\left(z^{*}\right)\right\Vert \leq\frac{2M\eta^{t}\left\Vert z-z^{*}\right\Vert }{\varpi\left(z^{*}\right)-M\eta^{t}\left\Vert z-z^{*}\right\Vert }\leq M\eta^{t}\left\Vert z-z^{*}\right\Vert ,\mbox{ for all }t\geq t_{1},
\]
which proves that $x^{*}$ is an exponentially stable point for
$V$.
\end{proof}

\section{Dynamical systems of diallelic gonosomal lethal genetic disorders }

A genetic disease is a disease caused by a mutation on a gene,
it is gonosomal (resp. autosomal) if the locus of the mutated
gene is gonosomal (resp. autosomal or pseudo-autosomal). A genetic
disease is said to be dominant or recessive if the mutant allele
is dominant or recessive. In gonosomal disease case, dominance
plays a role only in homogametic sex individuals, that is to
say carrying two similar gonosomes, heterogametic sex individuals
with the mutant allele will be sick in any event that the allele
is dominant or recessive. Finally an allele is lethal if it causes
the death of a carrier when this allele is dominantan d the death
of a homozygous individual when this allele is recessive.

\medskip{}

In what follows we consider a gonosomal diallelic genetic disease
with one lethal allele in the $XY$ sex determination system,
according to the dominant or recessive nature of the lethal allele
there are six types of gonosomal algebras corresponding to the
cases given in the table below:

\begin{center}
\begin{tabular}{cc|cc}
\cline{3-4}
 & \multicolumn{1}{c}{} & \multicolumn{2}{c}{\Male}\tabularnewline
 & \multicolumn{1}{c}{} & lethal & non-lethal\tabularnewline
\hline
 & lethal dominant & $\left(1,1\right)$ & $\left(1,2\right)$\tabularnewline
\Female & lethal recessive & $\left(2,1\right)$ & $\left(2,2\right)$\tabularnewline
 & non-lethal & $\left(3,1\right)$ & $\left(3,2\right)$\tabularnewline
\hline
\end{tabular}\medskip{}

\par\end{center}

In the following we denote by $X^{*}$ a gonosome $X$ bringing
the lethal allele.

\subsection{Asymptotic behavior of trajectories in the case (\Female \ lethal
dominant, \Male \ lethal)}

\textcompwordmark{}

In this case, genotypes $XX^{*}$, $X^{*}X^{*}$ and $X^{*}Y$
are lethal, only the two genotypes $XX$ and $XY$ are observed
in the population. The gonosomal algebra associated with this
situation is defined on the basis $\left(e,\widetilde{e}\right)$
by: $e\widetilde{e}=\gamma e+\left(1-\gamma\right)\widetilde{e}$,
it is stochastic if $0<\gamma<1$.
\begin{pro}
The gonosomal operator $W$ associated with the gonosomal algebra
$\mathbb{R}\left\langle e,\widetilde{e}\right\rangle $ defined
below has two fixed points : $\left(0,0\right)$ and $\left(\frac{1}{1-\gamma},\frac{1}{\gamma}\right)$,
$\gamma\neq0,1$.\end{pro}
\begin{proof}
For $z\in\mathbb{R}\left\langle e,\widetilde{e}\right\rangle $,
$z=xe+y\widetilde{e}$ the relation $z=\frac{1}{2}z^{2}$ is
equivalent to
\[
\begin{cases}
x & =\;\gamma xy\\
y & =\;\left(1-\gamma\right)xy
\end{cases}
\]
or
\[
\begin{cases}
\left(1-\gamma y\right)x & =\;0\\
\left(1-\left(1-\gamma\right)x\right)y & =\;0.
\end{cases}
\]
If $\gamma=0$ or $\gamma=1$ we get immediately $\left(x,y\right)=\left(0,0\right)$.
If $\gamma\neq0,1$ it is clear that if $x=0$ then $y=0$ and
if $x\neq0$ we deduce from the first equation $y=\frac{1}{\gamma}$
with this the second equation gives $x=\frac{1}{1-\gamma}$.\end{proof}
\begin{pro}
\label{prop:LR/R}Concerning operators $W$, $V$ associated
with the gonosomal stochastic algebra $\mathbb{R}\left\langle e,\widetilde{e}\right\rangle $:
$e\widetilde{e}=\gamma e+\left(1-\gamma\right)\widetilde{e}$,
$\left(0<\gamma<1\right)$, we have for any initial point $z^{\left(0\right)}=\left(x^{\left(0\right)},y^{\left(0\right)}\right)\in\mathbb{R}^{2}$:
\begin{eqnarray*}
\lim_{t\rightarrow\infty}W^{t}\left(z^{\left(0\right)}\right) & = & \begin{cases}
\left(0,0\right) & \mbox{if }\left|x^{\left(0\right)}y^{\left(0\right)}\right|<\tfrac{1}{\gamma\left(1-\gamma\right)}\medskip\\
\left(\tfrac{1}{1-\gamma},\tfrac{1}{\gamma}\right) & \mbox{if }\left|x^{\left(0\right)}y^{\left(0\right)}\right|=\tfrac{1}{\gamma\left(1-\gamma\right)}\medskip\\
+\infty & \mbox{if }\left|x^{\left(0\right)}y^{\left(0\right)}\right|>\tfrac{1}{\gamma\left(1-\gamma\right)}
\end{cases}\\
V^{t}\left(z^{\left(0\right)}\right) & = & \left(\gamma,1-\gamma\right),\quad\left(\forall t\geq1\right).
\end{eqnarray*}
\end{pro}
\begin{proof}
Let $z^{\left(t\right)}=W^{t}\bigl(z^{\left(0\right)}\bigr)=\bigl(x^{\left(t\right)},y^{\left(t\right)}\bigr)$.
We get
\[
\begin{cases}
x^{\left(t+1\right)} & =\;\gamma x^{\left(t\right)}y^{\left(t\right)}\\
y^{\left(t+1\right)} & =\left(1-\gamma\right)x^{\left(t\right)}y^{\left(t\right)}
\end{cases}
\]
from this we prove easily that for any $t\geq1$
\[
x^{\left(t\right)}=\frac{1}{1-\gamma}\left[\gamma\left(1-\gamma\right)x^{\left(0\right)}y^{\left(0\right)}\right]^{2^{t}}\mbox{ and }y^{\left(t\right)}=\frac{1}{\gamma}\left[\gamma\left(1-\gamma\right)x^{\left(0\right)}y^{\left(0\right)}\right]^{2^{t}},
\]
hence $\varpi\circ W^{t}\left(z^{\left(0\right)}\right)=\frac{1}{\gamma\left(1-\gamma\right)}\left[\gamma\left(1-\gamma\right)x^{\left(0\right)}y^{\left(0\right)}\right]^{2^{t}}$
and we use the result \emph{e}) of Proposition \ref{prop:W=000026O}.\end{proof}
\begin{rk}
In Proposition \ref{prop:Notions_Equilibrium_pts} the reciprocal
of the results are not true in general, indeed in the result
above the fixed point $\left(\frac{1}{1-\gamma},\frac{1}{\gamma}\right)$
is not stable for $W$ while its normalized $\left(\gamma,1-\gamma\right)$
is stable for $V$.
\end{rk}
\begin{flushleft}
\textbf{Application}: We consider a gonosomal diallelic gene
recessive lethal in females and lethal in males. We denote $0\leq\mu\leq1$
the mutation rate of the normal allele to the lethal in females
and $0\leq\eta\leq1$ the analogous rate in males. We assume
that in each individual mutation affects only one gonosome $X$
at a time, it follows that in gametogenesis we have: $XX\rightarrowtail\left(1-\mu\right)X+\mu X^{*}$,
$XY\rightarrowtail\frac{1-\eta}{2}X+\frac{\eta}{2}X^{*}+\frac{1}{2}Y$
and thus after reproduction $XX\times XY\rightarrowtail\frac{1-\eta}{2-\eta}XX+\frac{1}{2-\eta}XY$.
According to Proposition \ref{prop:LR/R} in each generation
the frequency distribution of a non-lethal allele is stationary
equal to $\left(\frac{1-\eta}{2-\eta},\frac{1}{2-\eta}\right)$,
we notice that it does not depend on the rate $\mu$ and the
frequency in females is lower than in males.
\par\end{flushleft}

\subsection{Asymptotic behavior of trajectories in the case (\Female \ lethal
recessive, \Male \ lethal)}

\textcompwordmark{}

In this case, genotypes $X^{*}X^{*}$ and $X^{*}Y$ are lethal,
thus we observe only $XX$, $XX^{*}$ and $XY$ types. Let $A$
be the gonosomal algebra of type $\left(2,1\right)$ with basis
$\left(e_{1},e_{2},e\right)$ defined by $e_{1}e=\gamma_{1}e_{1}+\gamma_{2}e_{2}+\gamma e$
and $e_{2}e=\delta_{1}e_{1}+\delta_{2}e_{2}+\delta e$ where
$\gamma_{i},\delta_{i}\geq0$ and $\gamma=1-\gamma_{1}-\gamma_{2},\delta=1-\delta_{1}-\delta_{2}$
with $\gamma,\delta\geq0$.

Let $W$ be the gonosomal operator $W$ associated to the gonosomal
algebra defined above. For $z^{\left(0\right)}=\bigl(x_{1}^{\left(0\right)},x_{2}^{\left(0\right)},y^{\left(0\right)}\bigr)$
consider  $z^{\left(t\right)}=W^{t}\bigl(z^{\left(0\right)}\bigr)$
where
\begin{equation}
W: \begin{cases}
x_{1}' & =\;\bigl(\gamma_{1}x_{1}+\delta_{1}x_{2}\bigr)y\medskip\\
x_{2}' & =\;\bigl(\gamma_{2}x_{1}+\delta_{2}x_{2}\bigr)y\medskip\\
y' & =\;\bigl(\gamma x_{1}+\delta x_{2}\bigr)y.
\end{cases}\label{eq:W_LR/L}
\end{equation}

\begin{pro} Let $Fix(W)$ be the set of fixed points of $W$.
In addition to the point $\left(0,0,0\right)$, the operator
$W$ has the following fixed points:

1) If $\gamma_{1}\delta_{2}-\gamma_{2}\delta_{1}=0$,

\begin{equation*}
Fix(W)\;=\;\begin{cases}
\left(\frac{1}{1-\gamma_1},0,\frac{1}{\gamma_{1}}\right), & if \ \ \gamma_{1}\ne0, \gamma_{1}\ne1, \delta_{2}=0, \gamma_2=0 \medskip\\
\left(0,\frac{1}{1-\delta_{2}}, \frac{1}{\delta_2}\right), & if \ \ \gamma_{1}=0, \delta_{2}\ne0, \delta_{2}\ne1, \delta_1=0 \medskip\\
\left(\frac{\gamma_{1}}{\left(\gamma_{1}+\gamma_{2}\right)\left(1-\gamma_{1}-\delta_{2}\right)},\frac{\gamma_{2}}{\left(\gamma_{1}+\gamma_{2}\right)\left(1-\gamma_{1}-\delta_{2}\right)},\frac{1}{\gamma_{1}+\delta_{2}}\right), & if \ \ \gamma_{1}\delta_{2}\neq0, \gamma_{1}+\delta_{2}\neq1, \gamma_{2}\delta_{1}\neq0.
\end{cases}
\end{equation*}

2) If $\gamma_{1}\delta_{2}-\gamma_{2}\delta_{1}\neq0$,\medskip{}

\begin{equation*}
Fix(W)\;=\;\begin{cases}
\left(\frac{\lambda}{1-\gamma_{1}},\frac{1-\lambda}{1-\gamma_{1}},\frac{1}{\gamma_{1}}\right),
\lambda\in\mathbb{R},  & if \ \ \gamma_{1}=\delta_{2}, \delta_{1}=0, \gamma_{2}=0 \medskip\\
\left(\frac{1}{1-\gamma_{1}},0,\frac{1}{\gamma_{1}}\right),
\left(0,\frac{1}{1-\delta_{2}},\frac{1}{\delta_{2}}\right),
 & if \ \ \gamma_{1}\neq\delta_{2}, \delta_{1}=0, \gamma_{2}=0 \medskip\\
\left(\frac{\gamma_{1}-\delta_{2}}{\left(1-\gamma_{1}\right)\left(\gamma_{1}+\gamma_{2}-\delta_{2}\right)},\frac{\gamma_{2}}{\left(1-\gamma_{1}\right)\left(\gamma_{1}+\gamma_{2}-\delta_{2}\right)},\frac{1}{\gamma_{1}}\right) , & if \ \ \delta_{1}=0, \gamma_{2}\neq0  \medskip\\
\left(0,\frac{1}{1-\delta_{2}},\frac{1}{\delta_{2}}\right), & if \ \ \delta_{1}=0, \gamma_{2}\neq0  \medskip\\
\left(\frac{\delta_{1}}{\left(1-\delta_{2}\right)\left(\delta_{1}+\delta_{2}-\gamma_{1}\right)},\frac{\delta_{2}-\gamma_{1}}{\left(1-\delta_{2}\right)\left(\delta_{1}+\delta_{2}-\gamma_{1}\right)},\frac{1}{\delta_{2}}\right) , & if \ \ \delta_{1}\neq0, \gamma_{2}=0  \medskip\\
\left(\frac{1}{1-\gamma_{1}},0,\frac{1}{\gamma_{1}}\right), & if \ \ \delta_{1}\neq0, \gamma_{2}=0  \medskip\\
\left(\frac{\delta_{1}y_{i}}{\left(\gamma\delta_{1}-\delta\gamma_{1}\right)y_{i}+\delta},\frac{1-\gamma_{1}y_{i}}{\left(\gamma\delta_{1}-\delta\gamma_{1}\right)y_{i}+\delta},y_{i}\right),
(i=1,2) &  if \ \ \delta_{1}\neq0, \gamma_{2}\neq0.
\end{cases}
\end{equation*}

where $y_{1}$ and $y_{2}$ are roots of $\left(\gamma_{1}\delta_{2}-\gamma_{2}\delta_{1}\right)y^{2}-\left(\gamma_{1}+\delta_{2}\right)y+1=0$.\end{pro}

\begin{proof}
Let us find the fixed points of $W$, for that we must solve
the system of equations:
\begin{equation}
\begin{cases}
x_{1} & =\;\bigl(\gamma_{1}x_{1}+\delta_{1}x_{2}\bigr)y\\
x_{2} & =\;\bigl(\gamma_{2}x_{1}+\delta_{2}x_{2}\bigr)y\\
y & =\;\bigl(\gamma x_{1}+\delta x_{2}\bigr)y
\end{cases}\label{eq:XLRec_PtFix}
\end{equation}

If $y=0$ we get the fixed point $\left(0,0,0\right)$.

If $y\neq0$ we write the system (\ref{eq:XLRec_PtFix}) in the
form:
\begin{equation}
\begin{cases}
\left(\gamma_{1}y-1\right)x_{1}+\left(\delta_{1}y\right)x_{2} & =\;0\\
\left(\gamma_{2}y\right)x_{1}+\left(\delta_{2}y-1\right)x_{2} & =\;0\\
\gamma x_{1}+\delta x_{2} & =\;1
\end{cases}\label{eq:XLRec_PtFix2}
\end{equation}
the determinant of the first two equations is necessarily zero,
thus
\begin{equation}
\left(\gamma_{1}\delta_{2}-\gamma_{2}\delta_{1}\right)y^{2}-\left(\gamma_{1}+\delta_{2}\right)y+1=0.\label{eq:XLRec_y}
\end{equation}
We consider two cases depending on the degree of the equation
(\ref{eq:XLRec_y}).

\textbf{Case-1}. If $\gamma_{1}\delta_{2}-\gamma_{2}\delta_{1}=0$ from (\ref{eq:XLRec_y})
we have $\gamma_{1}+\delta_{2}\neq0$, otherwise we have the unique fixed point $(0,0,0).$ Hence $y=\frac{1}{\gamma_{1}+\delta_{2}}$
then in (\ref{eq:XLRec_PtFix}) the first and second equations we get
 \begin{equation}\label{sys1}
\begin{cases}
\gamma_{2}x_{1}-\gamma_{1}x_{2} & =0\\
\delta_{2}x_{1}-\delta_{1}x_{2} & =0\\
\end{cases}
\end{equation}
Using this we get $\gamma x_{1}=\left(1-\gamma_{1}\right)x_{1}-\gamma_{1}x_{2}$
and $\delta x_{2}=\left(1-\delta_{2}\right)x_{1}-\delta_{2}x_{1}$
hence $\gamma x_{1}+\delta x_{2}=\left(1-\gamma_{1}-\delta_{2}\right)\left(x_{1}+x_{2}\right)=1$
consequently, if  $\gamma_{1}+\delta_{2}\neq1$ then $x_{1}+x_{2}=\frac{1}{1-\gamma_{1}-\delta_{2}}$. Of course, if  $\gamma_{1}+\delta_{2}=1$ then $\gamma x_{1}+\delta x_{2}=0$ and the system (\ref{eq:XLRec_PtFix2}) does not have any solution except $(0,0,0).$ So we consider the following subcases with condition $\gamma x_{1}+\delta x_{2}\ne0,1.$

Case 1.1. If $\gamma_{1}\ne0, \gamma_{1}\ne1, \delta_{2}=0, \gamma_{2}=0$, then from (\ref{sys1}) and taking into account $\gamma_{1}\delta_{2}-\gamma_{2}\delta_{1}=0$ we obtain the fixed point
$\left(\frac{1}{1-\gamma_1},0,\frac{1}{\gamma_{1}}\right);$

Case 1.2. If $\gamma_{1}=0, \delta_{2}\ne0, \delta_{2}\ne1, \delta_{1}=0$, while in the previous case, we obtain the next fixed point
$\left(0,\frac{1}{1-\delta_{2}}, \frac{1}{\delta_2}\right);$

Case 1.3. If $\gamma_{1}\ne0, \delta_{2}\ne0, \gamma_{1}+\delta_{2}\ne1, \gamma_2\ne0, \delta_{1}\ne0$, then from first equation of (\ref{sys1}) we get $x_2=\frac{\gamma_2 x_1}{\gamma_1}$ and then using $x_{1}+x_{2}=\frac{1}{1-\gamma_{1}-\delta_{2}}$ one has
$x_1=\frac{\gamma_{1}}{\left(\gamma_{1}+\gamma_{2}\right)\left(1-\gamma_{1}-\delta_{2}\right)},$ so $x_{2}=\frac{\gamma_{2}}{\left(\gamma_{1}+\gamma_{2}\right)\left(1-\gamma_{1}-\delta_{2}\right)}.$ Note that $\gamma_{1}\delta_{2}-\gamma_{2}\delta_{1}=0$, i.e., $\frac{\gamma_2}{\gamma_1}=\frac{\delta_2}{\delta_1}$ that we can get another equivalent fixed point form:
$x_{1}=\frac{\delta_{1}}{\left(\delta_{1}+\delta_{2}\right)\left(1-\gamma_{1}-\delta_{2}\right)}$
and $x_{2}=\frac{\delta_{2}}{\left(\delta_{1}+\delta_{2}\right)\left(1-\gamma_{1}-\delta_{2}\right)}$.

Note that for the other subcases the system (\ref{eq:XLRec_PtFix2}) has a unique trivial solution $(0,0,0).$

\medskip{}

\textbf{Case-2. }If $\gamma_{1}\delta_{2}-\gamma_{2}\delta_{1}\neq0$, the
discriminant of (\ref{eq:XLRec_y}) is $\Delta=\left(\gamma_{1}+\delta_{2}\right)^{2}-4\left(\gamma_{1}\delta_{2}-\gamma_{2}\delta_{1}\right)$
or $\Delta=\left(\gamma_{1}-\delta_{2}\right)^{2}+4\gamma_{2}\delta_{1}\geq0$.
Let $y_{1}$, $y_{2}$ be the roots of (\ref{eq:XLRec_y}).

If $\delta_{1}=0$ or $\gamma_{2}=0$ we have $\gamma_{1}\delta_{2}\ne0$ and the roots $y_{1}=\frac{1}{\gamma_{1}}$
and $y_{2}=\frac{1}{\delta_{2}}$.

Case 2.1. If $\delta_{1}=\gamma_{2}=0$ and $\gamma_{1}=\delta_{2}\ne1$
then $\gamma=\delta=1-\gamma_{1}$ and (\ref{eq:XLRec_PtFix2})
is reduced to $x_{1}+x_{2}=\frac{1}{1-\gamma_{1}}$ which results
to the fixed point $\left(\frac{\lambda}{1-\gamma_{1}},\frac{1-\lambda}{1-\gamma_{1}},\frac{1}{\gamma_{1}}\right)$
for any $\lambda\in\mathbb{R}$.

Case 2.2. If $\delta_{1}=\gamma_{2}=0$ and $\gamma_{1}\neq\delta_{2}$,
by using (\ref{eq:XLRec_PtFix2}) we get for the root $y_{1}$
the solution $\left(\frac{1}{1-\gamma_{1}},0,\frac{1}{\gamma_{1}}\right)$ with $\gamma_1\ne1$
and for $y_{2}$ the fixed point $\left(0,\frac{1}{1-\delta_{2}},\frac{1}{\delta_{2}}\right)$ with $\delta_2\ne1$.

Case 2.3. If $\delta_{1}=0$, $\gamma_{2}\neq0$ and $\gamma_{1}=\delta_{2}\ne1$ then from (\ref{eq:XLRec_PtFix2}) we get $\left(0,\frac{1}{1-\gamma_{1}},\frac{1}{\gamma_{1}}\right)$.

Case 2.4. If $\delta_{1}=0$, $\gamma_{2}\neq0$ and $\gamma_{1}\neq\delta_{2}$,
for the root $y_{1}=\frac{1}{\gamma_1}$ the system (\ref{eq:XLRec_PtFix2}) is
written
\[
\begin{cases}
\hspace{1.5cm}\gamma_{2}x_{1}+\left(\delta_{2}-\gamma_{1}\right)x_{2} & =\;0\\
\left(1-\gamma_{1}-\gamma_{2}\right)x_{1}+\left(1-\delta_{2}\right)x_{2} & =\;1
\end{cases}
\]
it follows the fixed point $\left(\frac{\gamma_{1}-\delta_{2}}{\left(1-\gamma_{1}\right)\left(\gamma_{1}+\gamma_{2}-\delta_{2}\right)},\frac{\gamma_{2}}{\left(1-\gamma_{1}\right)\left(\gamma_{1}+\gamma_{2}-\delta_{2}\right)},\frac{1}{\gamma_{1}}\right)$
with $\gamma_1\ne1,$ $\gamma_{1}+\gamma_{2}-\delta_{2}\ne0$ and for $y_{2}$ we get by (\ref{eq:XLRec_PtFix2}): $\left(0,\frac{1}{1-\delta_{2}},\frac{1}{\delta_{2}}\right)$ with $\delta_2\ne1$.

Case 2.5. If $\delta_{1}\neq0$, $\gamma_{2}=0$ and $\gamma_{1}=\delta_{2}\ne1$
we get from (\ref{eq:XLRec_PtFix2}) the solution $\left(\frac{1}{1-\gamma_{1}},0,\frac{1}{\gamma_{1}}\right)$.

Case 2.6. If $\delta_{1}\neq0$, $\gamma_{2}=0$ and $\gamma_{1}\neq\delta_{2}$,
for the root $y_{1}$ we get $\left(\frac{1}{1-\gamma_{1}},0,\frac{1}{\gamma_{1}}\right)$  with $\gamma_1\ne1$
and for $y_{2}$ the system (\ref{eq:XLRec_PtFix2}) becomes
\[
\begin{cases}
\left(\gamma_{1}-\delta_{2}\right)x_{1}+\delta_{1}x_{2} & =\;0\\
\left(1-\gamma_{1}\right)x_{1}+\left(1-\delta_{1}-\delta_{2}\right)x_{2} & =\;1
\end{cases}
\]
it follows the fixed point $\left(\frac{\delta_{1}}{\left(1-\delta_{2}\right)\left(\delta_{1}+\delta_{2}-\gamma_{1}\right)},\frac{\delta_{2}-\gamma_{1}}{\left(1-\delta_{2}\right)\left(\delta_{1}+\delta_{2}-\gamma_{1}\right)},\frac{1}{\delta_{2}}\right)$ with $\delta_2\ne1$ and $\delta_1+\delta_2-\gamma_1\ne0$.

Case 2.7. If $\delta_{1}\neq0$, $\gamma_{2}\neq0$ we have $\Delta>0$,
to each root $y_{i}$ of (\ref{eq:XLRec_y}) corresponds the
fixed point $\left(\frac{\delta_{1}y_{i}}{\left(\gamma\delta_{1}-\delta\gamma_{1}\right)y_{i}+\delta},\frac{1-\gamma_{1}y_{i}}{\left(\gamma\delta_{1}-\delta\gamma_{1}\right)y_{i}+\delta},y_{i}\right)$.
\end{proof}
In the following we consider the dynamical system $\left(z^{\left(t\right)}\right)_{t\geq0}$
generated by $W$ for a given initial point $z^{\left(0\right)}=\left(x_{1}^{\left(0\right)},x_{2}^{\left(0\right)},y^{\left(0\right)}\right)$,
we have $z^{\left(t\right)}=W^{t}\left(z^{\left(0\right)}\right)$
and $z^{\left(t\right)}=\left(x_{1}^{\left(t\right)},x_{2}^{\left(t\right)},y^{\left(t\right)}\right)$.
It is clear that if there is $t_{0}\geq0$ such as $y^{\left(t_{0}\right)}=0$
then by (\ref{eq:W_LR/L}) we have $W^{t}\left(z\right)=0$ for
all $t\geq t_{0}$. Now it is assumed that $y^{\left(t\right)}\neq0$
for all $t\geq0$.

To study the trajectories $\bigl(z^{\left(t\right)}\bigr)$ we
consider two cases depending on whether the set $\mathcal{E}_{z^{\left(0\right)}}=\left\{ t\in\mathbb{N}:x_{2}^{\left(t\right)}=0\right\} $
is infinite or finite.
\begin{lemma}
\label{lem:E_infinite} Let $W$ be the gonosomal operator defined 
by (\ref{eq:W_LR/L}) and $y^{\left(t\right)}\neq0$
for all $t\geq0$.

a) If $\gamma_{2}=0$, then the following are equivalent:

(i) $\mathcal{E}_{z^{\left(0\right)}}$ is infinite; (ii) $\mathbb{N}^{*}\subset\mathcal{E}_{z^{\left(0\right)}}$;
(iii) $x_{2}^{\left(1\right)}=0$.\medskip{}

b) If $\gamma_{2}\neq0$, then the following are equivalent:

(i) $\mathcal{E}_{z^{\left(0\right)}}$ is infinite; (ii) $\mathcal{E}_{z^{\left(0\right)}}=2\mathbb{N}\mbox{ or }\mathbb{N}\setminus2\mathbb{N}$;
(iii) $\begin{cases}
x_{1}^{\left(0\right)}=0, & x_{2}^{\left(1\right)}=x_{2}^{\left(3\right)}=0,\\
\mbox{or}\\
x_{1}^{\left(1\right)}=0, & x_{2}^{\left(0\right)}=x_{2}^{\left(2\right)}=0.
\end{cases}$\end{lemma}
\begin{proof}
\emph{a}) If we suppose $\gamma_{2}=0$, from (\ref{eq:W_LR/L})
we get: $\enskip x_{2}^{\left(t+1\right)}=\delta_{2}x_{2}^{\left(t\right)}y^{\left(t\right)}\enskip\left(*\right)$.

$\left(i\right)\Rightarrow\left(iii\right)$ Let $t_{0}$ be
the smallest element of $\mathcal{E}_{z^{\left(0\right)}}$,
if $t_{0}=0$ we deduce from $\left(*\right)$ that $x_{2}^{\left(t\right)}=0$
for all $t\geq0$. If $t_{0}\geq1$, from $0=x_{2}^{\left(t_{0}\right)}=\delta_{2}x_{2}^{\left(t_{0}-1\right)}y^{\left(t_{0}-1\right)}$,
$y^{\left(t_{0}-1\right)}\neq0$ and by minimality of $t_{0}$
we get $\delta_{2}=0$ but this implies $x_{2}^{\left(t\right)}=0$
for all $t\geq1$.

$\left(iii\right)\Rightarrow\left(i\right)$ If $x_{2}^{\left(1\right)}=0$
it is clear from $\left(*\right)$ that $x_{2}^{\left(t\right)}=0$
from all $t\geq1$.

$\left(ii\right)\Rightarrow\left(i\right)$ is trivial.\medskip{}

\emph{b}) If we have $\gamma_{2}\neq0$.

$\left(i\right)\Rightarrow\left(ii\right)$ Let $t_{0}$ be
the smallest element of $\mathcal{E}_{z^{\left(0\right)}}$,
from (\ref{eq:W_LR/L}) we have
\[
x_{1}^{\left(t_{0}+1\right)}=\gamma_{1}x_{1}^{\left(t_{0}\right)}y^{\left(t_{0}\right)},\quad x_{2}^{\left(t_{0}+1\right)}=\gamma_{2}x_{1}^{\left(t_{0}\right)}y^{\left(t_{0}\right)},\quad y^{\left(t_{0}+1\right)}=\gamma x_{1}^{\left(t_{0}\right)}y^{\left(t_{\text{0}}\right)}.
\]
And for any $m\geq1$ it exists $a_{m},b_{m},c_{m}\geq0$ such
as
\begin{equation}
\begin{cases}
x_{1}^{\left(t_{0}+m+1\right)}=\gamma^{2^{m-1}}a_{m}\left(x_{1}^{\left(t_{0}\right)}y^{\left(t_{0}\right)}\right)^{2^{m-1}} & \medskip\\
x_{2}^{\left(t_{0}+m+1\right)}=\gamma_{2}\gamma^{2^{m-1}}b_{m}\left(x_{1}^{\left(t_{0}\right)}y^{\left(t_{0}\right)}\right)^{2^{m-1}} & \medskip\\
y^{\left(t_{0}+m+1\right)}=\gamma^{2^{m-1}}c_{m}\left(x_{1}^{\left(t_{0}\right)}y^{\left(t_{0}\right)}\right)^{2^{m-1}}
\end{cases}\label{eq:Rel_m}
\end{equation}
with $a_{1}=\gamma_{1}$, $b_{1}=1$, $c_{1}=\gamma$ and
\[
a_{m+1}=c_{m}\left(\gamma_{1}a_{m}+\delta_{1}\gamma_{2}b_{m}\right),\quad b_{m+1}=c_{m}\left(a_{m}+\delta_{2}b_{m}\right),\quad c_{m+1}=c_{m}\left(\gamma a_{m}+\delta\gamma_{2}b_{m}\right).
\]

From $y^{\left(t\right)}\neq0$ for all $t\geq0$ and the third
equation of (\ref{eq:Rel_m}) we deduce $\gamma\neq0$, $x_{1}^{\left(t_{0}\right)}\neq0$
and $c_{m}\neq0$ for $m\geq1$. As $\mathcal{E}_{z^{\left(0\right)}}$
is infinite, there exists $m_{0}\geq3$ such as $x_{2}^{\left(t_{0}+m_{0}+1\right)}=0$,
thus we have $b_{m_{0}}=0$, from the relation giving $b_{m_{0}}$
it follows $a_{m_{0}-1}=\delta_{2}b_{m_{0}-1}=0\quad\left(*\right)$,
then $c_{m_{0}}=\delta\gamma_{2}c_{m_{0}-1}b_{m_{0}-1}$, as
$c_{m_{0}}\neq0$ we get $\delta\gamma_{2}b_{m_{0}-1}\neq0$
and with $\left(*\right)$ we get $\delta_{2}=0$. From $0=a_{m_{0}-1}=c_{m_{0}-2}\left(\gamma_{1}a_{m_{0}-2}+\delta_{1}\gamma_{2}b_{m_{0}-2}\right)$
we deduce $\gamma_{1}a_{m_{0}-2}=\delta_{1}\gamma_{2}b_{m_{0}-2}=0$.
If we suppose $\gamma_{1}\neq0$ then we get $a_{m_{0}-2}=0$ that
leads by recursively to the contradiction $a_{1}=0$. Thus we
have $\gamma_{1}=0$ and from (\ref{eq:W_LR/L}) we get
\[
\begin{cases}
x_{1}^{\left(t+1\right)} & =\;\delta_{1}x_{2}^{\left(t\right)}y^{\left(t\right)}\medskip\\
x_{2}^{\left(t+1\right)} & =\;\gamma_{2}x_{1}^{\left(t\right)}y^{\left(t\right)}\medskip\\
y^{\left(t+1\right)} & =\;\bigl(\left(1-\gamma_{2}\right)x_{1}^{\left(t\right)}+\left(1-\delta_{1}\right)x_{2}^{\left(t\right)}\bigr)y^{\left(t\right)}.
\end{cases}
\]
We can say that $\delta_{1}\neq0$ otherwise we would have $x_{1}^{\left(t\right)}=0$
for all $t\geq1$ hence $x_{2}^{\left(t\right)}=0$ for each
$t\geq2$ and $y^{\left(t\right)}=0$ for every $t\geq3$. Assuming
$t_{0}\geq2$, from $x_{2}^{\left(t_{0}\right)}=0$ we get $\gamma_{2}x_{1}^{\left(t_{0}-1\right)}y^{\left(t_{0}-1\right)}=0$
then $0=x_{1}^{\left(t_{0}-1\right)}=\delta_{1}x_{2}^{\left(t_{0}-2\right)}y^{\left(t_{0}-2\right)}$
hence $x_{2}^{\left(t_{0}-2\right)}=0$ which contradicts the
minimality of $t_{0}$. Therefore $t_{0}\leq1$, for $t_{0}=1$
we get $0=x_{2}^{\left(1\right)}=\gamma_{2}x_{1}^{\left(0\right)}y^{\left(0\right)}$
hence $x_{1}^{\left(0\right)}=0$ then we get $x_{2}^{\left(0\right)}\neq0$
otherwise $y^{\left(1\right)}=0$, next $x_{1}^{\left(2\right)}=\delta_{1}x_{2}^{\left(1\right)}y^{\left(1\right)}=0$
hence $x_{2}^{\left(3\right)}=\gamma_{2}x_{1}^{\left(2\right)}y^{\left(2\right)}=0$.
In the case $t_{0}=0$, we have $x_{2}^{\left(0\right)}=0$ hence
$x_{1}^{\left(1\right)}=0$ then $x_{2}^{\left(2\right)}=0$.

$\left(iii\right)\Rightarrow\left(ii\right)$ If $x_{1}^{\left(0\right)}=x_{2}^{\left(1\right)}=x_{2}^{\left(3\right)}=0$,
we have $0=x_{2}^{\left(1\right)}=\delta_{2}x_{2}^{\left(0\right)}y^{\left(0\right)}$,
since $x_{2}^{\left(0\right)}\neq0$ otherwise $y^{\left(1\right)}=0$
we get $\delta_{2}=0$. From this we deduce $x_{2}^{\left(2\right)}=\delta_{1}\gamma_{2}x_{2}^{\left(0\right)}y^{\left(0\right)}y^{\left(1\right)}$
and $0=x_{2}^{\left(3\right)}=\gamma_{1}\delta_{1}\gamma_{2}x_{2}^{\left(0\right)}y^{\left(0\right)}y^{\left(1\right)}y^{\left(2\right)}$
hence $\gamma_{1}\delta_{1}=0$, assuming $\delta_{1}=0$ we
get $x_{1}^{\left(1\right)}=0$ and the contradiction $y^{\left(2\right)}=0$,
thus we have $\delta_{1}\neq0$ and $\gamma_{1}=0$. Finally
we have $x_{2}^{\left(2t+1\right)}=\delta_{1}\gamma_{2}x_{2}^{\left(2t-1\right)}y^{\left(2\right)}$
for all $t\geq1$ and from $x_{2}^{\left(1\right)}=0$ we get
$\mathcal{E}_{z^{\left(0\right)}}=\mathbb{N}\setminus2\mathbb{N}$.

If $x_{1}^{\left(1\right)}=x_{2}^{\left(0\right)}=x_{2}^{\left(2\right)}=0$,
we have $0=x_{1}^{\left(1\right)}=\gamma_{1}x_{1}^{\left(0\right)}y^{\left(0\right)}$,
since $x_{1}^{\left(0\right)}\neq0$ otherwise $y^{\left(1\right)}=0$
we get $\gamma_{1}=0$. From $0=x_{2}^{\left(2\right)}=\delta_{2}x_{2}^{\left(1\right)}y^{\left(1\right)}$
and $x_{2}^{\left(1\right)}\neq0$ we get $\delta_{2}=0$. Then
for all $t\geq0$ we have $x_{2}^{\left(2t+2\right)}=\gamma_{2}x_{1}^{\left(2t+1\right)}y^{\left(2t+1\right)}=\delta_{1}\gamma_{2}x_{2}^{\left(2t\right)}y^{\left(2t\right)}y^{\left(2t+1\right)}$
, with this and $x_{2}^{\left(0\right)}=0$ we get $\mathcal{E}_{z^{\left(0\right)}}=2\mathbb{N}$.\end{proof}
\begin{thm}
Given any initial point $z^{\left(0\right)}\in\mathbb{R}^{3}$
such as $\mathcal{E}_{z^{\left(0\right)}}$ is infinite. For
the gonosomal operator (\ref{eq:W_LR/L}) we get:

a) if $\gamma_{2}=0$, then
\begin{eqnarray*}
\lim_{t\rightarrow\infty}W^{t}\left(z^{\left(0\right)}\right) & = & \begin{cases}
\left(0,0,0\right) & \mbox{if }\left|x_{1}^{\left(1\right)}y^{\left(1\right)}\right|<\frac{1}{\gamma_1(1-\gamma_1)}\medskip\\
\left(\frac{1}{1-\gamma_1},0,\frac{1}{\gamma_1}\right) & \mbox{if }\left|x_{1}^{\left(1\right)}y^{\left(1\right)}\right|=\frac{1}{\gamma_1(1-\gamma_1)}\medskip\\
+\infty & \mbox{if }\left|x_{1}^{\left(1\right)}y^{\left(1\right)}\right|>\frac{1}{\gamma_1(1-\gamma_1)}.
\end{cases}\\
V^{t+2}\left(z^{\left(0\right)}\right) & = & \left(\gamma_{1},0,1-\gamma_{1}\right),\quad\left(\forall t\geq0\right).
\end{eqnarray*}

b) if $\gamma_{2}\neq0$ and

\textbf{case 1}: if $x_{1}^{\left(0\right)}=0$, then
\begin{eqnarray*}
\lim_{t\rightarrow\infty}W^{t}\left(z^{\left(0\right)}\right) & =  \begin{cases}
\left(0,0,0\right) & \mbox{if }\left|x_{2}^{\left(0\right)}y^{\left(0\right)}\right|<\frac{1}{\sqrt[3]{\gamma_2\delta_1^2\gamma\delta^2}}\medskip\\
+\infty & \mbox{if }\left|x_{2}^{\left(0\right)}y^{\left(0\right)}\right|>\frac{1}{\sqrt[3]{\gamma_2\delta_1^2\gamma\delta^2}}.
\end{cases}\\
\end{eqnarray*}
if $\left|x_{2}^{\left(0\right)}y^{\left(0\right)}\right|=\frac{1}{\sqrt[3]{\gamma_2\delta_1^2\gamma\delta^2}}$ then $\forall t\geq0$
\begin{eqnarray*}
W^{2t+1}\left(z^{\left(0\right)}\right) & = \left(\frac{\delta_1}{\sqrt[3]{\gamma_2\delta_1^2\gamma\delta^2}},0,\frac{\delta}{\sqrt[3]{\gamma_2\delta_1^2\gamma\delta^2}} \right) \\
W^{2t+2}\left(z^{\left(0\right)}\right) & = \left(0, \frac{\gamma_2\delta_1\delta}{\sqrt[3]{\gamma_2\delta_1^2\gamma\delta^2}},\frac{\gamma\delta_1\delta}{\sqrt[3]{\gamma_2\delta_1^2\gamma\delta^2}} \right)\\
\end{eqnarray*}
and for any $z^{(0)}$ and $\forall t\geq0$
\begin{eqnarray*}
V^{2t+1}\left(z^{\left(0\right)}\right)  &= \left(\delta_{1},0,1-\delta_{1}\right),\\
V^{2t+2}\left(z^{\left(0\right)}\right)  &= \left(0,\gamma_{2},1-\gamma_{2}\right).
\end{eqnarray*}

\textbf{case 2}: if $x_{2}^{\left(0\right)}=0$,
\begin{eqnarray*}
\lim_{t\rightarrow\infty}W^{t}\left(z^{\left(0\right)}\right) & =  \begin{cases}
\left(0,0,0\right) & \mbox{if }\left|x_{1}^{\left(0\right)}y^{\left(0\right)}\right|<\frac{1}{\sqrt[3]{\gamma_2^2\delta_1\gamma^2\delta}}\medskip\\
+\infty & \mbox{if }\left|x_{1}^{\left(0\right)}y^{\left(0\right)}\right|>\frac{1}{\sqrt[3]{\gamma_2^2\delta_1\gamma^2\delta}}.
\end{cases}\\
\end{eqnarray*}
if $\left|x_{1}^{\left(0\right)}y^{\left(0\right)}\right|=\frac{1}{\sqrt[3]{\gamma_2^2\delta_1\gamma^2\delta}}$ then  $\forall t\geq0$
\begin{eqnarray*}
W^{2t+1}\left(z^{\left(0\right)}\right) & = \left(0,\frac{\gamma_2}{\sqrt[3]{\gamma_2^2\delta_1\gamma^2\delta}},\frac{\gamma}{\sqrt[3]{\gamma_2^2\delta_1\gamma^2\delta}} \right) \\
W^{2t+2}\left(z^{\left(0\right)}\right) & = \left(\frac{\delta_1\gamma_2\gamma}{\sqrt[3]{\gamma_2^2\delta_1\gamma^2\delta}},0,\frac{\delta\gamma_2\gamma}{\sqrt[3]{\gamma_2^2\delta_1\gamma^2\delta}} \right)\\
\end{eqnarray*}
and for any $z^{(0)}$ and $\forall t\geq0$ we have
\begin{eqnarray*}
V^{2t+1}\left(z^{\left(0\right)}\right) & = & \left(0,\gamma_{2},1-\gamma_{2}\right),\\
V^{2t+2}\left(z^{\left(0\right)}\right) & = & \left(\delta_{1},0,1-\delta_{1}\right).
\end{eqnarray*}

\end{thm}

\begin{proof}
\emph{a}) According to Lemma \ref{lem:E_infinite} we have
$x_{2}^{\left(t\right)}=0$ for $t\geq1$ and from $\gamma_{2}=0$ and with this (\ref{eq:W_LR/L}) becomes
for all $t\geq1$
\begin{equation}
\begin{cases}
x_{1}^{\left(t+1\right)} & =\;\gamma_{1}x_{1}^{\left(t\right)}y^{\left(t\right)}\medskip\\
y^{\left(t+1\right)} & =\;\left(1-\gamma_{1}\right)x_{1}^{\left(t\right)}y^{\left(t\right)}.
\end{cases}\label{eq:S1}
\end{equation}
We have $\gamma_{1}\neq0,1$ otherwise we would have $y^{\left(t\right)}=0$
for $t\geq3$. From (\ref{eq:S1}) we get
\[
\begin{cases}
x_{1}^{\left(t+2\right)} & =\;\gamma_{1}^{\,2^{t}}\left(1-\gamma_{1}\right)^{2^{t}-1}\left(x_{1}^{\left(1\right)}y^{\left(1\right)}\right)^{2^{t}}\medskip\\
y^{\left(t+2\right)} & =\;\gamma_{1}^{\,2^{t}-1}\left(1-\gamma_{1}\right)^{2^{t}}\left(x_{1}^{\left(1\right)}y^{\left(1\right)}\right)^{2^{t}},\quad t\geq0.
\end{cases}
\]
Since $0<\gamma_{1}<1$ we have $\lim_{t\rightarrow\infty}\gamma_{1}^{\,2^{t}}\left(1-\gamma_{1}\right)^{2^{t}}=0$
and with $\varpi\circ W^{t+2}\left(z^{\left(0\right)}\right)=\gamma_{1}^{\,2^{t}-1}\left(1-\gamma_{1}\right)^{2^{t}-1}\left(x_{1}^{\left(1\right)}y^{\left(1\right)}\right)^{2^{t}}$
we get the results of the proposition.\medskip{}

\emph{b}) We saw in the proof of Lemma \ref{lem:E_infinite}
that in this case we have for all $t\geq0$:
\[
\begin{cases}
x_{1}^{\left(t+1\right)} & =\;\delta_{1}x_{2}^{\left(t\right)}y^{\left(t\right)}\medskip\\
x_{2}^{\left(t+1\right)} & =\;\gamma_{2}x_{1}^{\left(t\right)}y^{\left(t\right)}\medskip\\
y^{\left(t+1\right)} & =\;\bigl(\gamma x_{1}^{\left(t\right)}+\delta x_{2}^{\left(t\right)}\bigr)y^{\left(t\right)}.
\end{cases}
\]
where $\gamma=1-\gamma_{2}$ and $\delta=1-\delta_{1}$.

\textbf{Case 1}: $x_{1}^{\left(0\right)}=0$. Then it is clear that $x_{2}^{\left(1\right)}=0.$

We have $x_{2}^{\left(0\right)}\neq0$ if not with $x_{1}^{\left(0\right)}=0$
we get $y^{\left(1\right)}=0$, therefore $x_{1}^{\left(1\right)}=\delta_{1}x_{2}^{\left(0\right)}y^{\left(0\right)}\neq0$.
We show that $x_{1}^{\left(2t\right)}=0$ and $x_{2}^{\left(2t+1\right)}=0$
for all $t\geq0$. Then for all $t\geq0$ we get:
\[
\begin{cases}
x_{1}^{\left(2t+1\right)} & =\;\delta_{1}x_{2}^{\left(2t\right)}y^{\left(2t\right)}\medskip\\
x_{2}^{\left(2t+2\right)} & =\;\gamma_{2}x_{1}^{\left(2t+1\right)}y^{\left(2t+1\right)}\medskip\\
y^{\left(2t+1\right)} & =\;\delta x_{2}^{\left(2t\right)}y^{\left(2t\right)}\\
y^{\left(2t+2\right)} & =\;\gamma x_{1}^{\left(2t+1\right)}y^{\left(2t+1\right)}.
\end{cases}
\]

It follows that

\[
\begin{cases}
x_{1}^{\left(2t+1\right)} & =\;\delta_{1}\left[\gamma_{2}\delta_{1}^{2}\gamma\delta^{2}\right]^{\nicefrac{\left(4^{t}-1\right)}{3}}\left(x_{2}^{\left(0\right)}y^{\left(0\right)}\right)^{4^{t}}\medskip\\
x_{2}^{\left(2t+2\right)} & =\;\gamma_{2}\delta_{1}\delta\left[\gamma_{2}^{2}\delta_{1}^{4}\gamma^{2}\delta^{4}\right]^{\nicefrac{\left(4^{t}-1\right)}{3}}\left(x_{2}^{\left(0\right)}y^{\left(0\right)}\right)^{2\times4^{t}}\medskip\\
y^{\left(2t+1\right)} & =\;\delta\left[\gamma_{2}\delta_{1}^{2}\gamma\delta^{2}\right]^{\nicefrac{\left(4^{t}-1\right)}{3}}\left(x_{2}^{\left(0\right)}y^{\left(0\right)}\right)^{4^{t}}\\
y^{\left(2t+2\right)} & =\;\gamma\delta_{1}\delta\left[\gamma_{2}^{2}\delta_{1}^{4}\gamma^{2}\delta^{4}\right]^{\nicefrac{\left(4^{t}-1\right)}{3}}\left(x_{2}^{\left(0\right)}y^{\left(0\right)}\right)^{2\times4^{t}}.
\end{cases}
\]

Since $y^{\left(t\right)}\neq0$ we get $\gamma_{2}\delta_{1}\gamma\delta\neq0$ and we can change the form of the last system:
\[
\begin{cases}
x_{1}^{\left(2t+1\right)} & =\frac{\delta_{1}}{\sqrt[3]{\gamma_{2}\delta_{1}^{2}\gamma\delta^{2}}}\left(x_{2}^{(0)}y^{(0)}\sqrt[3]{\gamma_{2}\delta_{1}^{2}\gamma\delta^{2}}\right)^{4^{t}}\medskip\\
x_{2}^{\left(2t+2\right)} & =\frac{\gamma_2\delta_{1}\delta}{\sqrt[3]{(\gamma_{2}\delta_{1}^{2}\gamma\delta^{2})^2}}\left(x_{2}^{(0)}y^{(0)}\sqrt[3]{\gamma_{2}\delta_{1}^{2}\gamma\delta^{2}}\right)^{2\times4^{t}}\medskip\\
y^{\left(2t+1\right)} & =\frac{\delta}{\sqrt[3]{\gamma_{2}\delta_{1}^{2}\gamma\delta^{2}}}\left(x_{2}^{(0)}y^{(0)}\sqrt[3]{\gamma_{2}\delta_{1}^{2}\gamma\delta^{2}}\right)^{4^{t}}\medskip\\
y^{\left(2t+2\right)} & =\frac{\gamma\delta_{1}\delta}{\sqrt[3]{(\gamma_{2}\delta_{1}^{2}\gamma\delta^{2})^2}}\left(x_{2}^{(0)}y^{(0)}\sqrt[3]{\gamma_{2}\delta_{1}^{2}\gamma\delta^{2}}\right)^{2\times4^{t}}.
\end{cases}
\]

Using  $0<\gamma_{2}\delta_{1}\gamma\delta<1$ we get the results of the proposition.

From
\begin{eqnarray*}
\varpi\circ W^{2t+1}\left(z^{\left(0\right)}\right) & = & \left[\gamma_{2}\delta_{1}^{2}\gamma\delta^{2}\right]^{\nicefrac{\left(4^{t}-1\right)}{3}}\left(x_{2}^{\left(0\right)}y^{\left(0\right)}\right)^{4^{t}}\\
\varpi\circ W^{2t+2}\left(z^{\left(0\right)}\right) & = & \delta_{1}\delta\left[\gamma_{2}^{2}\delta_{1}^{4}\gamma^{2}\delta^{4}\right]^{\nicefrac{\left(4^{t}-1\right)}{3}}\left(x_{2}^{\left(0\right)}y^{\left(0\right)}\right)^{2\times4^{t}},
\end{eqnarray*}
we deduce the values of $V^{2t+1}\left(z^{\left(0\right)}\right)$
and $V^{2t+2}\left(z^{\left(0\right)}\right)$.

\textbf{Case 2}: $x_{2}^{\left(0\right)}=0$. Then we get $x_{1}^{\left(1\right)}=0.$

We obtain $x_{1}^{\left(0\right)}\neq0$ if not with $x_{2}^{\left(0\right)}=0$
we get $y^{\left(1\right)}=0$, therefore $x_{2}^{\left(1\right)}=\gamma_2x_{1}^{\left(0\right)}y^{\left(0\right)}\neq0$.
Then for all $t\geq0$ we get $x_{1}^{\left(2t+1\right)}=0$ and $x_{2}^{\left(2t\right)}=0$ and
\[
\begin{cases}
x_{1}^{\left(2t+2\right)} & =\;\delta_{1}x_{2}^{\left(2t+1\right)}y^{\left(2t+1\right)}\medskip\\
x_{2}^{\left(2t+1\right)} & =\;\gamma_{2}x_{1}^{\left(2t\right)}y^{\left(2t\right)}\medskip\\
y^{\left(2t+2\right)} & =\;\delta x_{2}^{\left(2t+1\right)}y^{\left(2t+1\right)}\\
y^{\left(2t+1\right)} & =\;\gamma x_{1}^{\left(2t\right)}y^{\left(2t\right)}.
\end{cases}
\]

The results are derived from the previous case by exchanging
the roles of $x_{1}^{\left(t\right)}$ and $x_{2}^{\left(t\right)}$
at the same time as $\gamma_{2}$ with $\delta_{1}$ and $\gamma$ with $\delta$. \end{proof}
\begin{thm}
Given any initial point $z^{\left(0\right)}\in\mathbb{R}^{3}$
such as $\mathcal{E}_{z^{\left(0\right)}}$ is finite. For the
gonosomal operator (\ref{eq:W_LR/L}) we get:

(a) if $\gamma_{1}=\delta_{2}<1$ and $\gamma_{2}\delta_{1}=0$,
\begin{eqnarray*}
\lim_{t\rightarrow\infty}W^{t}\left(z^{\left(0\right)}\right)& = &\left(0,0,0\right)
\end{eqnarray*}

and for any $z^{\left(0\right)}\in S^{\,2}$,

\begin{eqnarray*}
\lim_{t\rightarrow+\infty}V^{t}\left(z^{\left(0\right)}\right) & = & \begin{cases}
\left(\gamma_1,0,\gamma\right) & \mbox{if }\gamma_{2}\neq0,\delta_{1}=0\medskip\\
\left(\frac{\gamma_1x_{1}^{(t_{0})}}{x_{1}^{(t_{0})}+x_{2}^{(t_{0})}},\frac{\delta_2x_{2}^{(t_{0})}}{x_{1}^{(t_{0})}+x_{2}^{(t_{0})}},\frac{\gamma x_{1}^{(t_{0})}+\delta x_2^{(t_0)}}{x_{1}^{(t_{0})}+x_{2}^{(t_{0})}}\right) & \mbox{if }\gamma_{2}=\delta_{1}=0,\medskip\\
\left(0,\delta_2,\delta\right) & \mbox{if }\gamma_{2}=0,\delta_{1}\neq0.
\end{cases}
\end{eqnarray*}

where $t_{0}=\max\left(\mathcal{E}_{z^{\left(0\right)}}\right)+1$.

(b) if $\gamma_{1}\neq\delta_{2}$ or $\gamma_{2}\delta_{1}\neq0$,

\begin{eqnarray*}
\lim_{t\rightarrow\infty}W^{t}\left(z^{\left(0\right)}\right)& = &\left(0,0,0\right)
\end{eqnarray*}

and for any $z^{\left(0\right)}\in S^{\,2}$,

\begin{eqnarray*}
\lim_{t\rightarrow+\infty}V^{t}(z^{(0)}) & =\left(\frac{\gamma_1+\delta_1u(\lambda_i)}{U(\lambda_i)}, \frac{u(\lambda_i)(\gamma_1+\delta_1u(\lambda_i))}{U(\lambda_i)}, \frac{\gamma+\delta u(\lambda_i)}{U(\lambda_i)}\right) &
\end{eqnarray*}
where $i=1$ if $|\lambda_1|<|\lambda_2|$ and $i=2$ if $|\lambda_1|>|\lambda_2|$,\\
and
\begin{eqnarray*}
\begin{cases}
U(\lambda_i) =\delta_1u(\lambda_i)^2+(\delta+\delta_1+\gamma_1)u(\lambda_i)+\gamma+\gamma_1, \medskip\\
u(\lambda_i) =\frac{\gamma_2x_1^{(t_0)}+(\delta_2-\lambda_i)x_2^{(t_0)}}{(\gamma_1-\lambda_i)x_1^{(t_0)}+\delta_1x_2^{(t_0)}},\medskip\\
\lambda_1=\frac{\gamma_1+\delta_2-\sqrt{(\gamma_{1}-\delta_{2})^{2}+4\gamma_{2}\delta_{1}}}{2}, \ \ \lambda_2=\frac{\gamma_1+\delta_2+\sqrt{(\gamma_{1}-\delta_{2})^{2}+4\gamma_{2}\delta_{1}}}{2}.
\end{cases}
\end{eqnarray*}

\end{thm}

\begin{proof}
Assume now that the set $\mathcal{E}_{z^{\left(0\right)}}$ is
finite. Let $t_{0}=\max\left(\mathcal{E}_{z^{\left(0\right)}}\right)+1$. We have $x_{2}^{\left(t\right)}\neq0$ for all $t\geq t_{0}$,
because $y^{\left(t\right)}\neq0$ for all $t\geq0$ it follows
from the second equation of (\ref{eq:W_LR/L}) that $\gamma_{2}x_{1}^{\left(t\right)}+\delta_{2}x_{2}^{\left(t\right)}\neq0$
for all $t\geq t_{0}$. From (\ref{eq:W_LR/L}) we get:
\[
\frac{x_{1}^{\left(t+1\right)}}{x_{2}^{\left(t+1\right)}}=\frac{\gamma_{1}x_{1}^{\left(t\right)}+\delta_{1}x_{2}^{\left(t\right)}}{\gamma_{2}x_{1}^{\left(t\right)}+\delta_{2}x_{2}^{\left(t\right)}},\quad\forall t\geq t_{0},
\]
taking $w^{\left(t\right)}=\frac{x_{1}^{\left(t\right)}}{x_{2}^{\left(t\right)}}$
for $t\geq t_{0}$, this is written as $w^{\left(t+1\right)}=f\left(w^{\left(t\right)}\right)$,
where $f\left(x\right)=\frac{\gamma_{1}x+\delta_{1}}{\gamma_{2}x+\delta_{2}}$.

Let $M=\left(\begin{array}{cc}
\gamma_{1} & \delta_{1}\\
\gamma_{2} & \delta_{2}
\end{array}\right)$, if $M^{t}=\left(\begin{array}{cc}
a_{t} & b_{t}\\
c_{t} & d_{t}
\end{array}\right)$ we verify that we have $f^{t}\left(x\right)=\frac{a_{t}x+b_{t}}{c_{t}x+d_{t}}$
for all $t\geq0$. The characteristic polynomial of $M$ is $\chi_{M}\left(X\right)=X^{2}-\left(\gamma_{1}+\delta_{2}\right)X+\left(\gamma_{1}\delta_{2}-\gamma_{2}\delta_{1}\right)$,
its discriminant is $\Delta=\left(\gamma_{1}-\delta_{2}\right)^{2}+4\gamma_{2}\delta_{1}\geq0$.
We have $\Delta=0$ if and only if $\gamma_{1}=\delta_{2}$ and
$\gamma_{2}\delta_{1}=0$.\medskip{}

(\emph{a}) The case $\Delta=0$.

Let $\lambda=\gamma_{1}$ the root of $\chi_{M}$, we have $\gamma_{1}<1$,
indeed if $\gamma_{1}=1$ then $\gamma=\gamma_{2}=0$ and $\delta=\delta_{1}=0$, thus
$\gamma=\delta=0$ which leads to the contradiction $y^{\left(t_{0}+1\right)}=0$.
Modulo $\chi_{M}$ we have for all $t\geq0$: $X^{t}\equiv t\lambda^{t-1}X-\left(t-1\right)\lambda^{t}$
hence $M^{t}=t\lambda^{t-1}M-\left(t-1\right)\lambda^{t}I_{2}$,
it follows that for any $m\geq1$ we get for $z^{\left(t_{0}\right)}\in\mathbb{R}^{3}$:

\begin{eqnarray*}
x_{1}^{\left(t_{0}+m\right)} & = & \lambda^{t_{0}+m-1}\left[\lambda x_{1}^{\left(t_{0}\right)}+\left(t_{0}+m\right)\delta_{1}x_{2}^{\left(t_{0}\right)}\right]y^{\left(t_{0}\right)}\\
x_{2}^{\left(t_{0}+m\right)} & = & \lambda^{t_{0}+m-1}\left[\left(t_{0}+m\right)\gamma_{2}x_{1}^{\left(t_{0}\right)}+\lambda x_{2}^{\left(t_{0}\right)}\right]y^{\left(t_{0}\right)}
\end{eqnarray*}

then
\begin{eqnarray*}
y^{\left(t_{0}+m\right)} & =y^{\left(t_{0}\right)} & \prod_{k=0}^{m-1}\left(\gamma x_{1}^{\left(t_{0}+k\right)}+\delta x_{2}^{\left(t_{0}+k\right)}\right).
\end{eqnarray*}
With $\lambda<1$, we get $\lim_{t\rightarrow+\infty}x_{1}^{\left(t\right)}=0$
and $\lim_{t\rightarrow+\infty}x_{2}^{\left(t\right)}=0$. Concerning
$y^{\left(t\right)}$, it is clear that there exists positive integer $k_0$ such that $\gamma x_1^{(t)}+\delta x_2^{(t)}<1$ for all $t>k_0.$ Finally we get $\lim_{t\rightarrow+\infty}y^{\left(t\right)}=0$.

For the study of the operator $V$, let $z^{\left(0\right)}\in S^{\,2}$,
we consider two cases.

Case 1: If $x_{1}^{\left(t_{0}+m\right)}\neq0$ for all $m\geq1$, then we get $$\frac{x_{2}^{\left(t_{0}+m\right)}}{x_{1}^{\left(t_{0}+m\right)}}=\frac{\left(t_{0}+m\right)\gamma_{2}x_{1}^{\left(t_{0}\right)}+\lambda x_{2}^{\left(t_{0}\right)}}{\lambda x_{1}^{\left(t_{0}\right)}+\left(t_{0}+m\right)\delta_{1}x_{2}^{\left(t_{0}\right)}}.$$
Thus we have
\[
\lim_{m\rightarrow+\infty}\frac{x_{2}^{\left(t_{0}+m\right)}}{x_{1}^{\left(t_{0}+m\right)}}=\begin{cases}
0 & \mbox{if }\gamma_{2}=0,\delta_{1}\neq0,\\
\frac{x_{2}^{\left(t_{0}\right)}}{x_{1}^{\left(t_{0}\right)}} & \mbox{if }\gamma_{2}=\delta_{1}=0,\\
+\infty & \mbox{if }\gamma_{2}\neq0,\delta_{1}=0.
\end{cases}
\]
and for $t\geq t_0+m+1$
\[
\lim_{t\rightarrow+\infty}\frac{y^{(t)}}{x_{1}^{(t)}}=\begin{cases}
\frac{\gamma}{\gamma_1} & \mbox{if }\gamma_{2}=0,\delta_{1}\neq0\\
\frac{\gamma x_{1}^{(t_{0})}+\delta x_{2}^{(t_{0})}}{\gamma_1x_{1}^{\left(t_{0}\right)}} & \mbox{if }\gamma_{2}=\delta_{1}=0,\\
+\infty & \mbox{if }\gamma_{2}=0,\delta_{1}\neq0.
\end{cases}
\]
and
\[
\lim_{t\rightarrow+\infty}\frac{y^{(t)}}{x_{2}^{(t)}}=\begin{cases}
+\infty & \mbox{if }\gamma_{2}=0,\delta_{1}\neq0\\
\frac{\gamma x_{1}^{(t_{0})}+\delta x_{2}^{(t_{0})}}{\delta_2x_{2}^{\left(t_{0}\right)}} & \mbox{if }\gamma_{2}=\delta_{1}=0,\\
\frac{\delta}{\delta_2} & \mbox{if }\gamma_{2}=0,\delta_{1}\neq0.
\end{cases}
\]

Using them and
$$\frac{x_{1}^{(t_{0}+m)}}{\varpi\circ W(z^{(t_{0}+m)})}=\frac{1}{1+\frac{x_{2}^{(t_{0}+m)}}{x_{1}^{(t_{0}+m)}}+\frac{y^{(t_{0}+m)}}{x_{1}^{(t_{0}+m)}}}, \ \ \ \ \frac{x_{2}^{(t_{0}+m)}}{\varpi\circ W(z^{(t_{0}+m)})}=\frac{1}{1+\frac{x_{1}^{(t_{0}+m)}}{x_{2}^{(t_{0}+m)}}+\frac{y^{(t_{0}+m)}}{x_{2}^{(t_{0}+m)}}},$$
$$\frac{y^{(t_{0}+m)}}{\varpi\circ W(z^{(t_{0}+m)})}=\frac{1}{1+\frac{x_{1}^{(t_{0}+m)}}{y^{(t_{0}+m)}}+\frac{x_2^{(t_{0}+m)}}{y^{(t_{0}+m)}}}$$
we get
\[
\lim_{m\rightarrow+\infty}\frac{x_{1}^{\left(t_{0}+m\right)}}{\varpi\circ W\left(z^{\left(t_{0}+m\right)}\right)}=\begin{cases}
\qquad \gamma_1 & \mbox{if }\gamma_{2}=0,\delta_{1}\neq0,\\
\frac{\gamma_1x_{1}^{(t_{0})}}{x_{1}^{\left(t_{0}\right)}+x_{2}^{\left(t_{0}\right)}} & \mbox{if }\gamma_{2}=\delta_{1}=0,\\
\qquad0 & \mbox{if }\gamma_{2}\neq0,\delta_{1}=0,
\end{cases}
\]

\[
\lim_{m\rightarrow+\infty}\frac{x_{2}^{\left(t_{0}+m\right)}}{\varpi\circ W\left(z^{\left(t_{0}+m\right)}\right)}=\begin{cases}
\qquad0 & \mbox{if }\gamma_{2}=0,\delta_{1}\neq0,\\
\frac{\delta_2x_{2}^{(t_{0})}}{x_{1}^{\left(t_{0}\right)}+x_{2}^{\left(t_{0}\right)}} & \mbox{if }\gamma_{2}=\delta_{1}=0,\\
\qquad\delta_2 & \mbox{if }\gamma_{2}\neq0,\delta_{1}=0,
\end{cases}
\]
and for $t\geq t_0+m+1$
\[
\lim_{m\rightarrow+\infty}\frac{y^{(t)}}{\varpi\circ W\left(z^{(t)}\right)}=\begin{cases}
\qquad\gamma & \mbox{if }\gamma_{2}=0,\delta_{1}\neq0,\\
\frac{\gamma x_1^{(t_0)}+\delta x_{2}^{(t_{0})}}{x_{1}^{\left(t_{0}\right)}+x_{2}^{\left(t_{0}\right)}} & \mbox{if }\gamma_{2}=\delta_{1}=0,\\
\qquad\delta & \mbox{if }\gamma_{2}\neq0,\delta_{1}=0.
\end{cases}
\]
Case 2: If there is $m_0\geq1$ such as $x_{1}^{\left(t_{0}+m_0\right)}=0$ then from $z^{\left(0\right)}\in S^{\,2}$ and by the formula for $x_{1}^{\left(t_{0}+m\right)}$
we get $x_{1}^{\left(t_{0}\right)}=0$ and $\delta_{1}=0$ thus
$x_{1}^{\left(t_{0}+m\right)}=0$ for every $m\geq1$ and we
get easily $\lim_{t\rightarrow+\infty}V^{t}\left(z^{\left(0\right)}\right)=\left(0,1,0\right)$.

\medskip{}
(\emph{b}) The case $\Delta>0$.

Let $\lambda_{1}<\lambda_{2}$ be the roots of $\chi_{M}$. Modulo
$\chi_{M}$ we have for all $t\geq0$: $$X^{t}\equiv\frac{\lambda_{2}^{t}-\lambda_{1}^{t}}{\lambda_{2}-\lambda_{1}}X-\lambda_{1}\lambda_{2}\frac{\lambda_{2}^{t-1}-\lambda_{1}^{t-1}}{\lambda_{2}-\lambda_{1}}$$
and with $\theta_{t}=\frac{\lambda_{2}^{t}-\lambda_{1}^{t}}{\lambda_{2}-\lambda_{1}}$
we have $M^{t}=\theta_{t}M-\lambda_{1}\lambda_{2}\theta_{t-1}I_{2}$
and thus for all $m\geq1$:
\begin{eqnarray*}
x_{1}^{\left(t_{0}+m\right)} & = & \left[\left(\gamma_{1}\theta_{t_{0}+m}-\lambda_{1}\lambda_{2}\theta_{t_{0}+m-1}\right)x_{1}^{\left(t_{0}\right)}+\delta_{1}\theta_{t_{0}+m}x_{2}^{\left(t_{0}\right)}\right]y^{\left(t_{0}\right)}\\
x_{2}^{\left(t_{0}+m\right)} & = & \left[\gamma_{2}\theta_{t_{0}+m}x_{1}^{\left(t_{0}\right)}+\left(\delta_{2}\theta_{t_{0}+m}-\lambda_{1}\lambda_{2}\theta_{t_{0}+m-1}\right)x_{2}^{\left(t_{0}\right)}\right]y^{\left(t_{0}\right)},
\end{eqnarray*}
hence
\begin{eqnarray*}
y^{\left(t_{0}+m\right)} & =y^{\left(t_{0}\right)} & \prod_{k=0}^{m-1}\left(\gamma x_{1}^{\left(t_{0}+k\right)}+\delta x_{2}^{\left(t_{0}+k\right)}\right).
\end{eqnarray*}

Let's prove that $|\lambda_1|<1$ and $|\lambda_2|<1.$ Since $\gamma_2<1-\gamma_1, \delta_1<1-\delta_2$ we get $0<\Delta=(\gamma_{1}-\delta_{2})^{2}+4\gamma_{2}\delta_{1}<(\gamma_{1}-\delta_{2})^{2}+4(1-\gamma_1)(1-\delta_2)=(\gamma_1+\delta_2-2)^2.$ From this we obtain $\lambda_2=\frac{\gamma_1+\delta_2+\sqrt{\Delta}}{2}<1$ and $\lambda_1=\frac{\gamma_1+\delta_2-\sqrt{\Delta}}{2}>\gamma_1+\delta_2-1>-1.$ So, $|\lambda_1|<1, \ \ |\lambda_2|<1$ and from this one has $\theta_t\rightarrow0$ as $t\rightarrow+\infty.$ Thus, we get $\lim_{t\rightarrow+\infty}x_{1}^{\left(t\right)}=\lim_{t\rightarrow+\infty}x_{2}^{\left(t\right)}=0$ and as previous case $\lim_{t\rightarrow+\infty}y^{\left(t\right)}=0.$

To study the operator $V$ for $z^{\left(0\right)}\in S^{\,2},$ by considering two cases as in (\emph{a}), we can get the proof of theorem.
\end{proof}
\textbf{Application}. \emph{Dosage compensation and X inactivation
	in mammals}.

In the XY-sex determination system, the female has two X chromosomes
and the male only one. The X chromosome carries many genes involved
in the functioning of cells, so in the absence of regulation,
a female would produce twice as many proteins coded by these
genes as a male, which would cause dysfunctions in these cells.
In the early stages of female embryo formation, a mechanism called
\emph{dosage compensation} (or \emph{lyonization}) inactivates
one of the two X chromosomes. The \emph{X} inactivation is controlled
by a short region on the \emph{X} chromosome called the \emph{X-inactivation
	center} (\emph{Xic}), the \emph{Xic} is active on the inactivated
\emph{X} chromosome. The \emph{Xic} site
is necessary and sufficient to cause the \emph{X} inactivation:
presence in a female embryo of one non-functional site \emph{Xic}
is lethal. 

If we denote by $X^{*}$ a gonosome $X$ carrying a non-functional
site \emph{Xic}, there are only three genotypes $XY$, $X^{*}Y$,
$XX$, thus the associated gonosomal algebra is of type $\left(1,2\right)$.
And in the definition of the gonosomal operator $W$, variables
$x_{1}^{\left(t\right)}$, $x_{2}^{\left(t\right)}$, $y^{\left(t\right)}$
are respectively associated to genotypes $XY$, $X^{*}Y$, $XX$.

Using Proposition \ref{prop:Iso-gonosomal-algebra} and \ref{prop:Equivalent_W}, Definition \ref{va2} and Proposition \ref{va7}, the results obtained in this section apply to this situation.

\subsection{Asymptotic behavior of trajectories in the case (\Female \ lethal
recessive, \Male \ non-lethal)}

\textcompwordmark{}

In this case only the genotype $X^{*}X^{*}$ is lethal, thus
we observe only the types $XX$, $XX^{*}$, $X^{*}Y$ and $XY$.
The general case of the dynamic system associated with this situation
is complex, for this reason we will study a simpler case motivated
by the following example.

In humans, hemophilia is a genetic disease caused by mutation
of a gene encoding coagulation factors and located on the $X$
gonosome. It is a gonosomal recessive lethal disease, meaning
that there are no homozygous women for the mutation, heterozygous
women have not hemophilia but are carriers and only men are met.
As many as one-third of hemophiliacs have no affected family
members, reflecting a high mutation rate ('de novo' mutations).\medskip{}

We denote $\mu$ (resp. $\eta$) where $0\leq\mu,\eta\leq1$,
the mutation rate from $X$ to $X^{*}$ in maternal (resp. paternal)
gametes. Assuming that during oogenesis and spermatogenesis mutation
when it occurs in a cell affects only one gonosome $X$ both
and considering that a mutated gene does not return to the wild
type, after gametogenesis we observe the following rates:

\begin{center}
\begin{tabular}{ll}
$XX\rightarrowtail\left(1-\mu\right)X+\mu X^{*}$, & \quad{}$XY\rightarrowtail\frac{1-\eta}{2}X+\frac{\eta}{2}X^{*}+\frac{1}{2}Y$,$\medskip$\tabularnewline
$XX^{*}\rightarrowtail\frac{1-\mu}{2}X+\frac{1+\mu}{2}X^{*}$, & \quad{}$X^{*}Y\rightarrowtail\frac{1}{2}X^{*}+\frac{1}{2}Y$.\tabularnewline
\end{tabular}
\par\end{center}

Therefore after breeding the genotypes frequency distribution
is given in the following Punnet square:

\begin{center}
\begin{tabular}{clrrr}
$XX\times XY$ & $\rightarrowtail$\quad{} $\frac{\left(1-\mu\right)\left(1-\eta\right)}{2-\mu\eta}XX,$ & $\frac{\mu+\eta-2\mu\eta}{2-\mu\eta}XX^{*},$ & $\frac{1-\mu}{2-\mu\eta}XY,$ & $\frac{\mu}{2-\mu\eta}X^{*}Y$$\medskip$\tabularnewline
$XX\times X^{*}Y$ & $\rightarrowtail$ & $\frac{1-\mu}{2-\mu}XX^{*},$ & $\frac{1-\mu}{2-\mu}XY,$ & $\frac{\mu}{2-\mu}X^{*}Y$$\medskip$\tabularnewline
$XX^{*}\times XY$ & $\rightarrowtail$\quad{} $\frac{\left(1-\mu\right)\left(1-\eta\right)}{4-\left(1+\mu\right)\nu}XX,$ & $\frac{1+\mu-2\mu\eta}{4-\left(1+\mu\right)\eta}XX^{*},$ & $\frac{1-\mu}{4-\left(1+\mu\right)\eta}XY,$ & $\frac{1+\mu}{4-\left(1+\mu\right)\eta}X^{*}Y$$\medskip$\tabularnewline
$XX^{*}\times X^{*}Y$ & $\rightarrowtail$ & $\frac{1-\mu}{3-\mu}XX^{*},$ & $\frac{1-\mu}{3-\mu}XY,$ & $\frac{1+\mu}{3-\mu}X^{*}Y$\tabularnewline
\end{tabular}
\par\end{center}

Algebra associated with this situation is the gonomal $\mathbb{R}$-algebra
of type $\left(2,2\right)$, with basis $\left(e_{1},e_{2}\right)\cup\left(\widetilde{e}_{1},\widetilde{e}_{2}\right)$
and commutative multiplication table:
\begin{eqnarray*}
e_{1}\widetilde{e}_{1} & = & \tfrac{\left(1-\mu\right)\left(1-\eta\right)}{2-\mu\eta}e_{1}+\tfrac{\mu+\eta-2\mu\eta}{2-\mu\eta}e_{2}+\tfrac{1-\mu}{2-\mu\eta}\widetilde{e}_{1}+\tfrac{\mu}{2-\mu\eta}\widetilde{e}_{2}\\
e_{1}\widetilde{e}_{2} & = & \tfrac{1-\mu}{2-\mu}e_{2}+\tfrac{1-\mu}{2-\mu}\widetilde{e}_{1}+\tfrac{\mu}{2-\mu}\widetilde{e}_{2}\\
e_{2}\widetilde{e}_{1} & = & \tfrac{\left(1-\mu\right)\left(1-\eta\right)}{4-\left(1+\mu\right)\nu}e_{1}+\tfrac{1+\mu-2\mu\eta}{4-\left(1+\mu\right)\nu}e_{2}+\tfrac{1-\mu}{4-\left(1+\mu\right)\nu}\widetilde{e}_{1}+\tfrac{1+\mu}{4-\left(1+\mu\right)\nu}\widetilde{e}_{2}\\
e_{2}\widetilde{e}_{2} & = & \tfrac{1-\mu}{3-\mu}e_{2}+\tfrac{1-\mu}{3-\mu}\widetilde{e}_{1}+\tfrac{1+\mu}{3-\mu}\widetilde{e}_{2}
\end{eqnarray*}
not mentioned products are zero.

\medskip{}

From (\ref{eq:Op-W}) the dynamical system associated with this
algebra is:

\begin{equation}
W_{\mu,\eta}:\begin{cases}
\begin{array}{ccllll}
x_{1}' & = & \frac{\left(1-\mu\right)\left(1-\eta\right)}{2-\mu\eta}x_{1}y_{1} &  & +\frac{\left(1-\mu\right)\left(1-\eta\right)}{4-\left(1+\mu\right)\eta}x_{2}y_{1}\\
x_{2}' & = & \frac{\mu+\eta-2\mu\eta}{2-\mu\eta}x_{1}y_{1} & +\frac{1-\mu}{2-\mu}x_{1}y_{2} & +\frac{1+\mu-2\mu\eta}{4-\left(1+\mu\right)\eta}x_{2}y_{1} & +\frac{1-\mu}{3-\mu}x_{2}y_{2}\\
y_{1}' & = & \frac{1-\mu}{2-\mu\eta}x_{1}y_{1} & +\frac{1-\mu}{2-\mu}x_{1}y_{2} & +\frac{1-\mu}{4-\left(1+\mu\right)\eta}x_{2}y_{1} & +\frac{1-\mu}{3-\mu}x_{2}y_{2}\\
y_{2}' & = & \frac{\mu}{2-\mu\eta}x_{1}y_{1} & +\frac{\mu}{2-\mu}x_{1}y_{2} & +\frac{1+\mu}{4-\left(1+\mu\right)\eta}x_{2}y_{1} & +\frac{1+\mu}{3-\mu}x_{2}y_{2}
\end{array}\end{cases}\label{eq:SD-Hemophilia}
\end{equation}

\begin{pro}
Fixed points for the operators $W_{1,1}$ and $W_{1,\eta}$ is
$\left(0,0,0,0\right)$ and for $W_{\mu,1}$ are $\left(0,0,0,0\right)$
and $\left(0,\frac{3-\mu}{2},\frac{3-\mu}{2},\frac{\left(1+\mu\right)\left(3-\mu\right)}{2\left(1-\mu\right)}\right)$.\end{pro}
\begin{proof}
Let $z=\left(x_{1},x_{2},y_{1},y_{2}\right)$, consider the equation
$z=W_{\mu,\eta}\left(z\right)$.

a) If $\mu=\eta=1$ we get immediately in (\ref{eq:SD-Hemophilia}):
$x_{1}=x_{2}=y_{1}=0$ and thus $y_{2}=0$.

b) If $\mu=1$ and $\eta\neq1$, in (\ref{eq:SD-Hemophilia})
with $\mu=1$ we get $x_{1}=y_{1}=0$ it follows that $x_{2}=y_{2}=0$.

c) If $\mu\neq1$ and $\eta=1$, fixed points $\left(x_{1},x_{2},y_{1},y_{2}\right)$
of operator $W_{\mu,1}$ verify
\begin{equation}
\begin{cases}
\begin{array}{ccl}
x_{1} & = & 0\\
x_{2} & = & \frac{1-\mu}{3-\mu}x_{2}\left(y_{1}+y_{2}\right)\\
y_{1} & = & \frac{1-\mu}{3-\mu}x_{2}\left(y_{1}+y_{2}\right)\\
y_{2} & = & \frac{1+\mu}{3-\mu}x_{2}\left(y_{1}+y_{2}\right),
\end{array}\end{cases}\label{eq:Pt_Fix_Wm,1}
\end{equation}

If $y_{1}+y_{2}=0$ we have $x_{1}=x_{2}=y_{1}=y_{2}=0$. It
is assumed that $y_{1}+y_{2}\neq0$, by summing the last two
equations of (\ref{eq:Pt_Fix_Wm,1}) we get $y_{1}+y_{2}=\frac{2}{3-\mu}x_{2}\left(y_{1}+y_{2}\right)$
thus $x_{2}=\frac{3-\mu}{2}$ then $y_{1}=\frac{1-\mu}{2}\left(y_{1}+y_{2}\right)$
and $y_{2}=\frac{1+\mu}{2}\left(y_{1}+y_{2}\right)$ hence $y_{1}=\frac{1-\mu}{1+\mu}y_{2}$
it follows $y_{1}+y_{2}=\frac{2}{1+\mu}y_{2}$ and with the equation
giving $y_{2}$ in (\ref{eq:Pt_Fix_Wm,1}) we get $y_{2}=\frac{\left(1+\mu\right)\left(3-\mu\right)}{2\left(1-\mu\right)}$
hence $y_{1}=\frac{3-\mu}{2}$. Finally the fixed points of $W_{\mu,1}$
are: $\left(0,0,0,0\right)$ and $\left(0,\frac{3-\mu}{2},\frac{3-\mu}{2},\frac{\left(1+\mu\right)\left(3-\mu\right)}{2\left(1-\mu\right)}\right)$.\end{proof}

\begin{pro}
For all $z=\left(x_{1},x_{2},y_{1},y_{2}\right)\in\mathbb{R}^{4}$
and $0\leq\mu,\eta\leq1$ we have:

a) $W_{1,1}^{n}\left(z\right)=0$ for every $n\geq2$.

b) $W_{1,\eta}^{n}\left(z\right)=0$ for each $n\geq3$.

c) $\lim_{n\rightarrow\infty}W_{\mu,1}^{n}\left(z\right)=\begin{cases}
0 & \mbox{if }\left|\frac{x_{1}}{2-\mu}+\frac{x_{2}}{3-\mu}\right|\cdot\left|y_{1}+y_{2}\right|\leq\frac{1}{\left(1-\mu\right)^{2}}\medskip\\
+\infty & \mbox{if }\left|\frac{x_{1}}{2-\mu}+\frac{x_{2}}{3-\mu}\right|\cdot\left|y_{1}+y_{2}\right|>\frac{1}{\left(1-\mu\right)^{2}}.
\end{cases}$ \smallskip{}

And for the normalized gonosomal operator $V_{\mu,1}$ defined
by $W_{\mu,1}$ we have:
\[
V_{\mu,1}^{n}\left(z\right)=\left(0,\frac{1-\mu}{3-\mu},\frac{1-\mu}{3-\mu},\frac{1+\mu}{3-\mu}\right),\quad\forall n\geq1.
\]
\end{pro}
\begin{proof}
\emph{a}) If $\mu=\eta=1$, the system (\ref{eq:SD-Hemophilia})
becomes:
\[
\begin{cases}
x_{1}' & =\;x_{2}'=y_{1}'=0\\
y_{2}' & =\;\left(x_{1}+x_{2}\right)\left(y_{1}+y_{2}\right)
\end{cases}
\]
in other words, there are no more females in the first generation
and the population died in the second generation.\smallskip{}

\emph{b}) If $\mu=1$ and $\eta\neq1$, the system (\ref{eq:SD-Hemophilia})
is written:
\[
\begin{cases}
\begin{array}{ccllll}
x_{1}' & = & 0\\
x_{2}' & = & \frac{1-\eta}{2-\eta}x_{1}y_{1} &  & +\frac{1-\eta}{2-\eta}x_{2}y_{1}\\
y_{1}' & = & 0\\
y_{2}' & = & \frac{1}{2-\eta}x_{1}y_{1} & +x_{1}y_{2} & +\frac{1}{2-\eta}x_{2}y_{1} & +x_{2}y_{2}
\end{array}\end{cases}
\]
for $z=\left(x_{1},x_{2},y_{1},y_{2}\right)$ we find $z^{\left(2\right)}=\bigl(0,0,0,\left(\frac{1-\eta}{2-\eta}\right)^{2}\left(x_{1}+x_{2}\right)^{2}y_{1}^{2}\bigr)$
and thus $z^{\left(3\right)}=\left(0,0,0,0\right)$, the population
goes out to the third generation.\smallskip{}

\emph{c}) With $\mu\neq1$ and $\eta=1$, the system (\ref{eq:SD-Hemophilia})
becomes:
\[
\begin{cases}
\begin{array}{ccl}
x_{1}' & = & 0\\
x_{2}' & = & \left(\frac{1-\mu}{2-\mu}x_{1}+\frac{1-\mu}{3-\mu}x_{2}\right)\left(y_{1}+y_{2}\right)\\
y_{1}' & = & \left(\frac{1-\mu}{2-\mu}x_{1}+\frac{1-\mu}{3-\mu}x_{2}\right)\left(y_{1}+y_{2}\right)\\
y_{2}' & = & \left(\frac{\mu}{2-\mu}x_{1}+\frac{1+\mu}{3-\mu}x_{2}\right)\left(y_{1}+y_{2}\right).
\end{array}\end{cases}
\]

If for $z=\left(x_{1},x_{2},y_{1},y_{2}\right)\in S^{2,2}$ and
$n\geq0$, we put $W_{\mu,1}^{n}\left(z\right)=\bigl(x_{1}^{\left(n\right)},x_{2}^{\left(n\right)},y_{1}^{\left(n\right)},y_{2}^{\left(n\right)}\bigr)$,
we show that
\begin{eqnarray}
x_{1}^{\left(n+1\right)} & = & 0\nonumber \\
x_{2}^{\left(n+1\right)} & = & 2^{2^{n}-1}\frac{\left(1-\mu\right)^{2^{n+1}-1}}{\left(3-\mu\right)^{2^{n}-1}}\left(\frac{x_{1}}{2-\mu}+\frac{x_{2}}{3-\mu}\right)^{2^{n}}\left(y_{1}+y_{2}\right)^{2^{n}}\nonumber \\
y_{1}^{\left(n+1\right)} & = & x_{2}^{\left(n+1\right)}\label{eq:W_m,1}\\
y_{2}^{\left(n+1\right)} & = & 2^{2^{n}-1}\left(1+\mu\right)\frac{\left(1-\mu\right)^{2^{n+1}-2}}{\left(3-\mu\right)^{2^{n}-1}}\left(\frac{x_{1}}{2-\mu}+\frac{x_{2}}{3-\mu}\right)^{2^{n}}\left(y_{1}+y_{2}\right)^{2^{n}}.\nonumber
\end{eqnarray}
We have $\frac{2}{3}<\frac{2}{3-\mu}<1$, $x_{2}^{\left(n+1\right)}=\frac{1}{1-\mu}\left(\frac{2}{3-\mu}\right)^{2^{n}-1}\left[\left(1-\mu\right)^{2}\left(\frac{x_{1}}{2-\mu}+\frac{x_{2}}{3-\mu}\right)\left(y_{1}+y_{2}\right)\right]^{2^{n}}$,
$y_{1}^{\left(n+1\right)}=x_{2}^{\left(n+1\right)}$ and $y_{2}^{\left(n+1\right)}=\frac{1-\mu}{1+\mu}x_{2}^{\left(n+1\right)}$
from which we deduce the limit values of $W_{\mu,1}^{n}$.

From (\ref{eq:W_m,1}) we get $\varpi\circ W_{\mu,1}^{n}\left(z\right)=2^{2^{n}-1}\frac{\left(1-\mu\right)^{2^{n+1}-2}}{\left(3-\mu\right)^{2^{n}-2}}\left(\frac{x_{1}}{2-\mu}+\frac{x_{2}}{3-\mu}\right)^{2^{n}}\left(y_{1}+y_{2}\right)^{2^{n}}$,
for all $n\geq1$ and by normalization of terms given by (\ref{eq:W_m,1})
we get the $V_{\mu,1}^{n}$ components $\left(0,\frac{1-\mu}{3-\mu},\frac{1-\mu}{3-\mu},\frac{1+\mu}{3-\mu}\right)$
for all $n\geq1$.
\end{proof}
Now in what follows we assume that $\mu,\eta\neq1$.

\begin{pro}
For any $z=(x_1,x_2,y_1,y_2)\in S^{2,2}$
and $0\leq\mu,\eta\leq1$ the trajectory $\{z^{(n)}\}$ tends to the fixed point $0$ exponentially fast.
\end{pro}

\begin{proof} It is clear that $x_1^{(n)}\geq0, x_2^{(n)}\geq0, y_1^{(n)}\geq0, y_2^{(n)}\geq0$ for any $n\geq1$.  We choose the function
$F(z)=(x_1+x_2)(y_1+y_2)$ and show that $F(z)$ is a Lyapunov function for (\ref{eq:SD-Hemophilia}). Consider
$$F(z')=(x'_1+x'_2)(y'_1+y'_2)=(x'_1+x'_2+y'_1+y'_2)(y'_1+y'_2)-(y'_1+y'_2)^2.$$
Using b) of Proposition \ref{pro} we get that $y'_1+y'_2\leq\frac{1}{4}$ and from (\ref{omega}) we obtain
$$F(z')=(x_1+x_2)(y_1+y_2)(y'_1+y'_2)-(y'_1+y'_2)^2= (y'_1+y'_2) F(z)-(y'_1+y'_2)^2 \leq F(z).$$
Thus, the sequence $F(z^{(n)})$ is decreasing and bounded from below with 0, so it has a limit, i.e. it is a Lyapunov function. In addition, from b) of Proposition \ref{pro}

$$F(z')=(x'_1+x'_2)(y'_1+y'_2)\leq\left(\frac{1}{4}\right)^2,$$
on the other hand, $F(z')=x_1^{(2)}+x_2^{(2)}+y_1^{(2)}+y_2^{(2)}\leq\left(\frac{1}{4}\right)^2$ and from this we get $x_1^{(2)}+x_2^{(2)}\leq\left(\frac{1}{4}\right)^2, y_1^{(2)}+y_2^{(2)}\leq\left(\frac{1}{4}\right)^2.$ Thus, $F(z^{(2)})\leq\left(\frac{1}{4}\right)^{2^2}$ and so on. Hence, one has  $F(z^{(n)})\leq\left(\frac{1}{4}\right)^{2^n}$ for any $n\ge1$ and this guarantees that the limit of $F(z^{(n)})$ converges to 0. In addition, from $F(z^{(n)})=(x_1^{(n)}+x_2^{(n)})(y_1^{(n)}+y_2^{(n)})=x_1^{(n+1)}+x_2^{(n+1)}+y_1^{(n+1)}+y_2^{(n+1)}$ we obtain that 
$$0\leq x_1^{(n+1)}\leq F(z^{(n)}), \, 0\leq x_2^{(n+1)}\leq F(z^{(n)}), \, 0\leq y_1^{(n+1)}\leq F(z^{(n)}), \, 0\leq y_2^{(n+1)}\leq F(z^{(n)})$$
which completes the proof of the proposition.
\end{proof}

\end{document}